\documentclass[11pt]{amsart}
\pdfoutput=1
\usepackage{amssymb}
\usepackage{amsmath}
\usepackage{amsfonts}
\usepackage[usenames]{color}
\usepackage{graphicx}
\usepackage{array}

\makeatletter
 
 \@addtoreset{equation}{section}
\makeatother

\newtheorem{theorem}{Theorem}[section]
\newtheorem{lemma}[theorem]{Lemma}

\theoremstyle{definition}
\newtheorem{definition}[theorem]{Definition}

\newtheorem{proposition}[theorem]{Proposition}

\newtheorem{corollary}[theorem]{Corollary}

\theoremstyle{remark}
\newtheorem{remark}[theorem]{Remark}

\numberwithin{equation}{section}

\newcommand{\R}{\mathbb{R}}
\newcommand{\Z}{\mathbb{Z}}

\newcommand{\de}{\delta }
\newcommand{\cp}{\mathcal{P}}

\begin{document}

\title{On Legendrian Graphs}

\author{Danielle O'Donnol}
\address{Department of Mathematics, Rice University, Houston, Texas 77005}
\email{Danielle.S.Odonnol@rice.edu}

\author{Elena Pavelescu}
\address{Department of Mathematics, Rice University, Houston, Texas 77005}
\email{Elena.Pavelescu@rice.edu}

%    General info

\date{\today}


\begin{abstract}
We investigate Legendrian graphs in $(\R^3, \xi_{std})$. 
We extend the classical invariants, Thurston-Bennequin number and rotation number to Legendrian graphs. 
We prove that a graph can be Legendrian realized with all its cycles Legendrian unknots with $tb=-1$ and $rot=0$ if and only if it does not contain $K_4$ as a minor. 
We show that the pair $(tb, rot)$ does not characterize a Legendrian graph up to Legendrian isotopy if the graph contains a cut edge or a cut vertex.
For the lollipop graph the pair $(tb,rot)$ determines two Legendrian classes and for the handcuff graph it determines four Legendrian classes. 

\end{abstract}

\maketitle

%------------------ Introduction  --------------------------------------

\section{Introduction}\label{intro}
In this paper we begin the systematic study of Legendrian graphs in $\R^3$ with the standard contact structure. These are embedded graphs that are everywhere tangent to the contact planes.  
Legendrian graphs have appeared naturally in several important contexts in the study of contact manifolds.  
They are used in Giroux's proof of existence of open book decompositions compatible with a given contact structure, see \cite{G}. 
Legendrian graphs also appear in the proof of Eliashberg's and Fraser's result which says that in a tight contact structure the unknot is determined up to Legendrian isotopy by the invariants $tb$ and $rot$, see \cite{EF}. 
 Yet no study of Legendrian graphs, until now, has been undertaken.  We remedy this by establishing the foundations for what we expect will be a very rich field.  

A \emph{spatial graph} is an embedding of a graph into $\R^3.$  An \emph{abstract graph} is a set of vertices together with a set of edges between them, without any specified embedding.  We will throughout this paper refer to an abstract graph as simply a graph.  In Section \ref{graphs}, we show there is no obstruction to having a Legendrian realization of any spatial graph. 
 We extend the classical invariants Thurston-Bennequin number, $tb$, and rotation number, $rot$, from Legendrian knots to Legendrian graphs.  

In \cite{M}, Mohnke proved that the Borromean rings and the Whithead link cannot be represented by Legendrian links of \emph{trivial unknots}, that is unknots with $tb=-1$ and $rot=0$.  The trivial unknot is the one unknot among all unknots that attains the maximal Thurston-Bennequin number of its topological class.
As an application of our invariants, we ask which graphs admit Legendrian embeddings with all cycles trivial unknots. 
In Section \ref{maxtb}, we give a full characterization of these graphs, in the form of the following:

\begin{theorem}  A graph $G$ admits a Legendrian embedding in  $(\R^3, \xi_{std})$ with all its cycles trivial unknots if and only if $G$ does not contain $K_4$ as a minor.  
\end{theorem}

The proof of this theorem relies partly on the fact that the trivial unknot has an odd Thurston-Bennequin number, that is $tb=-1$.  We prove the reverse implication in more generality, for $L_{odd},$ which represents the set of topological knot classes with odd maximal Thurston-Bennequin number.  
\begin{theorem}
Let $G$  be a graph that contains $K_4$ as a minor. 
There does not exist a Legendrian realization of $G$ such that all its cycles are knots in $L_{odd} $ realizing their maximal Thurston-Bennequin number.
\end{theorem}

It is known that certain Legendrian knots and links are determined by the invariants $tb$ and $rot$: the unknot (see \cite{EF}), the torus knots and the figure eight knot (see \cite{EH}), the links consisting of an unknot and a cable of that unknot (see \cite{DG}). 
In Section \ref{Classification}, we ask what types of spatial graphs are classified up to Legendrian isotopy by the pair $(tb, rot)$.  
We show that the pair $(tb, rot)$ does not classify graphs which contain both cycles and either cut vertices or cut edges, independent of the chosen topological class. 
This means that even uncomplicated graphs carry more information than the set of knots represented by their cycles. 

In order to have a classification by the pair $(tb, rot)$, we must first narrow to a specific topological class.  We investigate \emph{topological planar graphs}, that is an embedded graph that is ambient isotopic to a plannar embedding.  
Not to be confused with \emph{planar graphs}, which refers to an abstract graph that has a planar embedding.  In the case of the handcuff graph, for topological planar graphs, we prove there are exactly four Legendrian realizations for each pair $(tb, rot)$.  For the lollipop graph, for topological planar graphs, we prove there are exactly two Legendrian realizations for each pair $(tb, rot)$.

\subsection{Aknowledgements}
The authors would like to thank Tim Cochran for his support and interest in the project.  
They would also like to thank John Etnyre for helpful conversations.

%--------------Background---------------

\section{Background}\label{backg}

\subsection{Spatial Graphs}\label{spatialgrph}
A \emph{spatial graph} is an embedding $f$ of a graph $G$ into $\R^3$ (or $S^3$), also called a \emph{spatial embedding} or \emph{graph embedding}.  
We remind the reader that an abstract graph, one without an embedding well be referred to as simply a graph.  We will be considering spatial graphs in $\R^3$ throughout this paper.  
%A \emph{flat vertex graph} or \emph{rigid vertex graph} is a spatial graph $f(G),$ where at each of the vertices $v$ of $f(G)$ there is a neighborhood $N_v$ of $v$ and a small disc $D_v$ such that $f(G)\cap N_v\subset D_v$.
%In other words, a flat vertex graph is one where the order of the edges around each vertex is fixed.  
Two spatial graphs $f(G)$ and $\bar{f}(G)$ are \emph{ambient isotopic} if there exits an isotopy $h_t:\R^3\to \R^3$ such that $h_0=id$ and  $h_1(f(G))=\bar{f}(G).$  
%Two flat vertex graphs $f(G)$ and $\bar{f}(G)$ are \emph{ambient isotopic} if there exits an isotopy $h_t:\R^3\to \R^3$ such that $h_0=id,$  $h_1(f(G))=\bar{f}(G),$ and $h_t(f(G))$ are flat vertex graphs for all $t\in[0,1]$. 
Similar to knots, there is a set of Reidemeister moves for spatial graphs, described by Kauffman in \cite{K}.  
%We will predominately be working with the less restricted topological objects, spatial graphs, however it will be useful to be aware of rigid vertex graphs as well.  \\

Here we remind the reader of some basic graph theoretic terminology.  
We are considering the most general of graphs, so there can be  \emph{multi-edges} 
(edges that go between the same pair of vertices), and \emph{loops} (edges that connect a vertex to itself).  
The \emph{valence} of a vertex is the number of endpoints of edges at the given vertex. Two vertices are \emph{adjacent} if 
there is an edge between them.  The complete graph on $n$ vertices, denoted $K_n$, is the graph with $n$ vertices were 
every vertex is adjacent to every other vertex, with exactly one edge between each pair of vertices.  Two edges are \emph{adjacent} if they share a vertex.  
A graph $H$ is a \emph{minor} of $G$ if $H$ can be obtained from a subgraph of $G$ by a finite number  of edge contractions. 

Since we are considering topological questions about Legendrian graphs throughout this paper, it is important to be aware of intrinsic properties of graphs.  
A property is called \emph{intrinsic} if every embedding of $G$ in $\R^3$ (or $S^3$) has the property.  
A graph $G$ is \emph{minor minimal} with respect to a property if $G$ has the property but no minor of $G$ has the property.  Such properties are characterized by their full set of minor minimal graphs.

A spatial graph is said to \emph{contain a knot (or link)} if the knot (or link) appears as a subgraph of $G$.  
A graph, $G$, is \emph{intrinsically knotted} if every embedding of $G$ into $\R^3$ (or $S^3$) contains a nontrivial knot. 
A link $L$ is \emph{split} if there is an embedding of a $2-$sphere $F$ in $\R^3 \setminus L$ such that each component of  $\R^3\setminus F$  contains at least one component of $L$.
A graph, $G$, is \emph{intrinsically linked} if every embedding of
$G$ into $\R^3$ (or $S^3$) contains a nonsplit link.   
The combined work of Conway and Gordon \cite{CG}, Sachs \cite{S}, and  Robertson, Seymour, and Thomas \cite{RST} fully characterized intrinsically linked graphs.
They showed that the Petersen family is the complete set of minor minimal intrinsically linked graphs.  That is no minor of the 
Petersen family is intrinsically linked and any intrinsically linked graph contains a graph in the Petersen family as a minor. 
Unlike the intrinsically linked graphs, the set of intrinsically knotted graphs has not been characterized.  
However, there are many graphs that are known to be minor minimal intrinsically knotted.  
As consequence of how intrinsically linked graphs were characterized, it is known that all intrinsically knotted graphs are intrinsically linked, see \cite{RST}.  

As we examine Legendrian graphs and their properties, it is important to be aware of what  topological properties (without restricting to Legendrian embeddings) are known to be intrinsic to the graph.  
These issues will be present when considering Theorem \ref{thmG} and Corollary \ref{unknot}.

\subsection{Legendrian Knots}
Let M be an oriented 3-manifold and $\xi$ a 2-plane field on $M$.
Then $\xi$ is a \textit{contact structure} on $M$ if  $\xi =\ker \alpha$ for some $1-$form $\alpha$ on $M$ satisfying $\alpha\wedge d\alpha > 0.$

On $\mathbb{R}^3$, the $1-$form $\alpha= dz-ydx$ defines a contact structure  called the standard contact structure, $\xi_{std}$. 
There is a diffeomeorphism of $\mathbb{R}^3$ taking the standard contact structure, $\xi_{std}$, to the symmetric contact structure, $\xi_{sym},$ given in cylindrical coordinates by $\alpha_1= dz+r^2d\theta$. In this paper, we switch between $\xi_{std}$ and  $\xi_{sym}$  when convenient. 
Darboux's theorem says that any contact structure on a manifold $M$ is locally diffeomorphic to $\xi_{std}$.

A curve $\gamma \subset (M, \xi)$ is called \textit{Legendrian} if for all $p\in\gamma$ and $\xi_p$ the contact plane at $p$,  $T_p\gamma \in \xi_p$.

In $(\R^3, \xi_{std})$, Legendrian curves, in particular knots and links, are studied via projections, and a common projection is the front projection (on $xz-$plane).
The Legendrian condition implies that $y =dz/dx$ (i.e. the $y-$coordinate can be recovered as the slope in the $xz-$plane), and so front projections of Legendrian knots do not have vertical tangencies. 
Figure \ref{unknot_trefoil} shows two front projections of Legendrian unknots and that of a Legendrian right-handed trefoil.  
Since the positive $y$-axis points inside the page, at each crossing the overstrand is always the one with smaller slope.

% ----------- Figure Unknot and Trefoil -----------------

\begin{figure}[htpb!]
\begin{center}
\begin{picture}(385, 72)
\put(0,0){\includegraphics{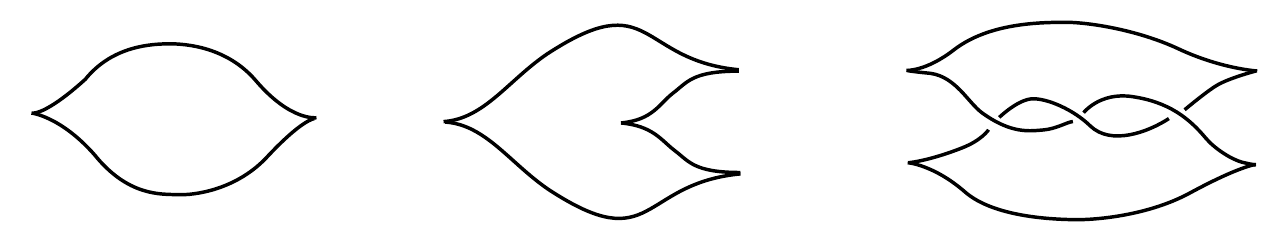}}
\end{picture}
\caption{Legendrian unknots and Legendrian right-handed trefoil}\label{unknot_trefoil}
\end{center}
\end{figure}

Apart from the topological knot class, there are two classical invariants of Legendrian knots, the Thurston-Bennequin number, $tb$, and the rotation number, $rot$. 
The Thurston-Bennequin number measures the amount of twisting of the contact planes along the knot and does not depend on the chosen orientation of $K$. 
To compute the Thurston-Bennequin number of a Legendrian knot $K$ consider a non-zero vector field $v$ transverse to $\xi$, take $K'$ the push-off of $K$ in the direction of $v$, and define $tb(K):= lk(K,K').$ If $K$ is null-homologous, $tb(K)$ measures the twisting of the contact framing on $K$ with respect to the Seifert framing.  
For a Legendrian knot $K$,  $tb(K)$ can be computed from its front projection $\tilde{K}$ as $$tb(K) = \rm{writhe(\tilde{K})} - \frac{1}{2}\# cusps (\tilde{K}).$$

The rotation number, $rot(K)$, is only defined for oriented null-homologous knots, so assume $K$ is oriented and $K=\partial \Sigma$, where $\Sigma\subset M$ is an embedded oriented surface. 
The contact planes when restricted to $\Sigma$ form a trivial $2-$dimensional bundle, and the trivialization of $\xi | _{\Sigma}$ induces a trivialization on  $\xi | _L = L\times \mathbb{R}^2$.  
Let $v$ be a non-zero vector field tangent to $K$ pointing in the direction of the orientation on $K$.  
The winding number of $v$ about the origin with respect to this trivialization is the rotation number of $K$, $rot(K)$ . 
In $\R^3$, the vector fields $d_1=\frac{\partial}{\partial y}$ and $d_2=-y\frac{\partial }{\partial z}-\frac{\partial}{\partial x}  $ define a positively oriented trivialization for $\xi_{std}$. 
Therefore, $rot(K)$ can be computed by counting with sign (  + for counterclockwise and $-$ for clockwise) how many times the positive tangent vector to $K$ crosses $ d_1$ as we travel once around $K$.  
The tangent vector aligns with with one of the vectors $ d_1$ or $-d_1$ at the points corresponding to cusps in the front projection, $\tilde{K}$, and one can check that $$rot(K) = \frac{1}{2}(\rm{\# down \hspace{0.05in}  cusps -\# up   \hspace{0.05in} cusps})(\tilde{K}).$$

Given a Legendrian knot $K$, Legendrian knots in the same topological class as $K$ can be obtained by \textit{stabilizations}.
A strand of $K$ in the front projection of $K$ is replaced by one of the zig-zags as in Figure \ref{stabilizations}. 
The stabilization is said to be positive if down cusps are introduced and negative if up cusps are introduced. 
The Legendrian isotopy type of $K$ changes through stabilization and so do the Thurston-Bennequin number and rotation number : $tb(S_{\pm}(K)) = tb(K) -1$ and $rot(S_{\pm}(K)) = rot(K) \pm 1$.

% ----------- stabilizations -----------------

\begin{figure}[htpb!]
\begin{center}
\begin{picture}(300, 115)
\put(0,0){\includegraphics[width=4.3in]{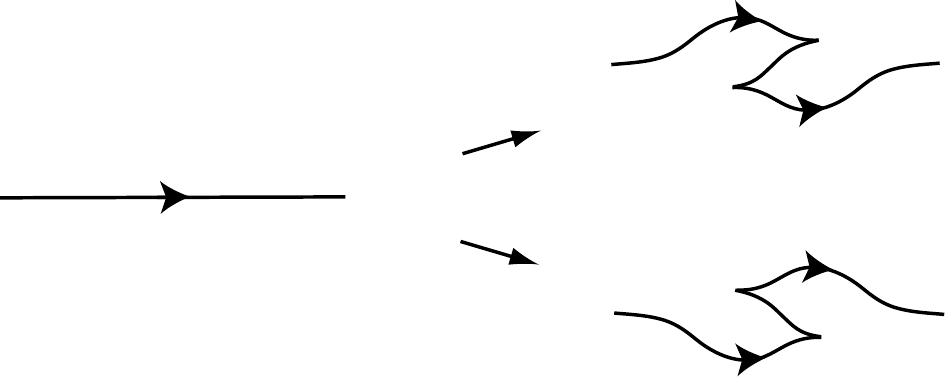}}
\put(53,67){$K$}
\put(135,85){\small $S_+(K)$}
\put(135,28){\small $S_-(K)$}
\end{picture}
\caption{Positive and negative stabilizations in the front
projection.}\label{stabilizations}
\end{center}
\end{figure}
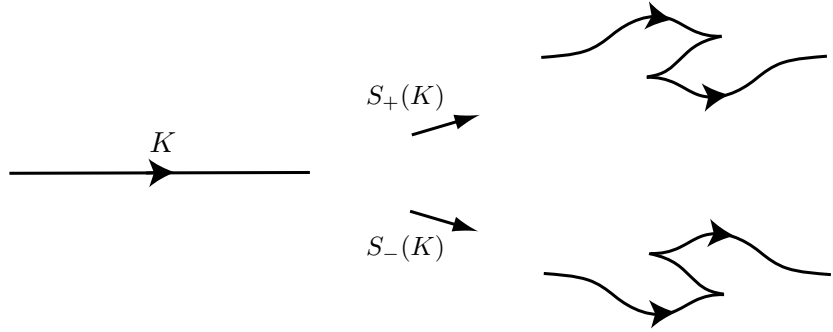

It is known, \cite{El}, that if $K$ is a  null-homologous Legendrian knot in  $(\R^3, \xi_{std})$ and $\Sigma$ is a Seifert surface for $K$ then $$tb(K) + |rot (K)| \leq - \chi(\Sigma)$$

In particular,  if $K$ is the unknot then $tb(K) \le -1$ and if $T$ is the right hand trefoil knot then $tb(T)\le 1$. 
Both inequalities are sharp and are realized for the first and third knots in Figure \ref{unknot_trefoil}.\\

%------------ Legendrian Graphs -----------------

\section{Legendrian graphs}\label{graphs}

\begin{definition}A \textit{Legendrian graph} in a contact 3-manifold $(M, \xi)$ is a graph embedded in such a way that all its edges are Legendrian curves that are non-tangent to each other at the vertices.  When the graph $G$ must be specified this will also be referred to as a \emph{Legendrian embedding of $G$} or a \emph{Legendrian realization of $G$}.  
\end{definition}

It should be noted for Legendrian graphs, if all edges around a vertex are oriented outward, then no two tangent vectors at the vertex coincide in the contact plane.
However, two tangent vectors may have the same direction but different orientations resulting in a smooth arc through the vertex.
It is a result of this structure that the order of the edges around a vertices is not changed under Legendrian isotopy.  

\begin{proposition}
Given any embedded graph $G$ in $\mathbb{R}^3$ there exists a Legendrian realization of $G$ in $(\R^3, \xi_{std})$.
\label{leggraph}
\end{proposition}
\begin{proof}
Denote the vertices of $G$ by $v_1, v_2, \dots, v_n$, and fix these points. 
Every point $v_i$ has an $\epsilon-$neighborhood $U_i$ contactomorphic to a neighborhood of the origin in $(\R^3, \xi_{sym})$.
Via this diffeomorphism, the contact plane $\xi_{v_i}$ is identified with the plane $z=0$ at the origin.
Near each vertex $v_i$, we modify $G$ through ambient isotopy, such that the edges incident with $v_i$ are segments which lie in the contact plane $\xi_{v_i}$ and are identified with $\theta$-constant segments in the plane $z=0$. 
The edges of $G$ are thus Legendrian near each vertex.

Consider $e$ an edge between two vertices $v_i$ and $v_j$ and let $e$ be identified with the $\theta_i-$ray near $v_i$ and with the $\theta_j-$ray near $v_j$.  
Denote by  $p_i\in  U_i$ and $p_j\in  U_j$ the two points which are identified with $(\frac{\epsilon}{2}, \theta_i, 0)$ and $(\frac{\epsilon}{2}, \theta_j, 0)$. 
Denote by $e_{p_i}$ the Legendrian segment between $v_i$ and $p_i$ identified with the segment $0\le r\le \frac{\epsilon}{2}, \theta=\theta_i, z=0$, and by $e_{p_j}$ the Legendrian segment between $v_j$ and $p_j$ identified with the segment $0\le r\le \frac{ \epsilon}{2}, \theta=\theta_j, z=0$. 
Denote by $e_{ij}$ the arc of $e$ between $p_i$ and $p_j$. 
We can $C^0-$approximate $e_{ij}$ by a Legendrian arc $\tilde{e_{ij}}$, in such a way that the union of arcs  $e_{p_i}\cup \tilde{e_{ij}}\cup e_{p_j}$ is a $C^1-$ curve. 
We do this by choosing a $C^0-$close approximation of the front projection of $e_{ij}$ by a regular curve $\overline{e_{ij}}$ with no vertical tangencies,  isolated cusps, and such that  at each point $p\in \overline{e_{ij}}$ the slope of   $\overline{e_{ij}}$ is close to the $y$-coordinate of the point on $e_{ij}$ projecting to $p$. 
We can do this while keeping the endpoints $p_i$ and $p_j$ fixed and in such a way that the Legendrian lift of $e_{ij}$, $\tilde{e_{ij}}$, coincides with $e$ near $p_i$ and $p_j$. 
For more details see \cite{Ge}, section 3.3.1. 
The curve  $e_{p_i}\cup \tilde{e_{ij}}\cup e_{p_j}$ is a Legendrian curve which is $C^0-$close to $e$. 

By making  $\overline{e_{ij}}$ have only  transverse intersections, no self-intersection in the approximating curve are introduced, thus $e_{p_i}\cup \tilde{e_{ij}}\cup e_{p_j}$ is an embedded arc.
Additionally, we can avoid introducing any non-trivial knotting. 

We approximate all other edges in the same way, and we obtain a Legendrian $C^0-$ approximation of $G$ (i.e. a Legendrian graph topologically ambient isotopic to $G$).
 
 \end{proof}

Legendrian graphs can also be studied via front projections and two generic front projections of a Legendrian graph are related by Reidemeister moves I, II and III together with three moves given by the mutual position of vertices and edges, see \cite{BI}. See Figure \ref{moves}.

%---- Figure ---- Reidemeister moves + + + -----------

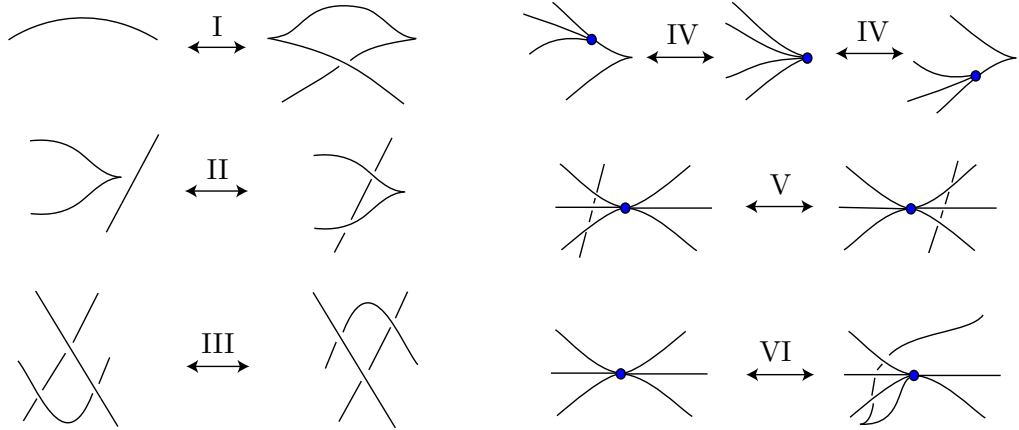
\begin{figure}[htpb!]
\begin{center}
\begin{picture}(360, 188)
\put(0,0){\includegraphics[width=5.5in]{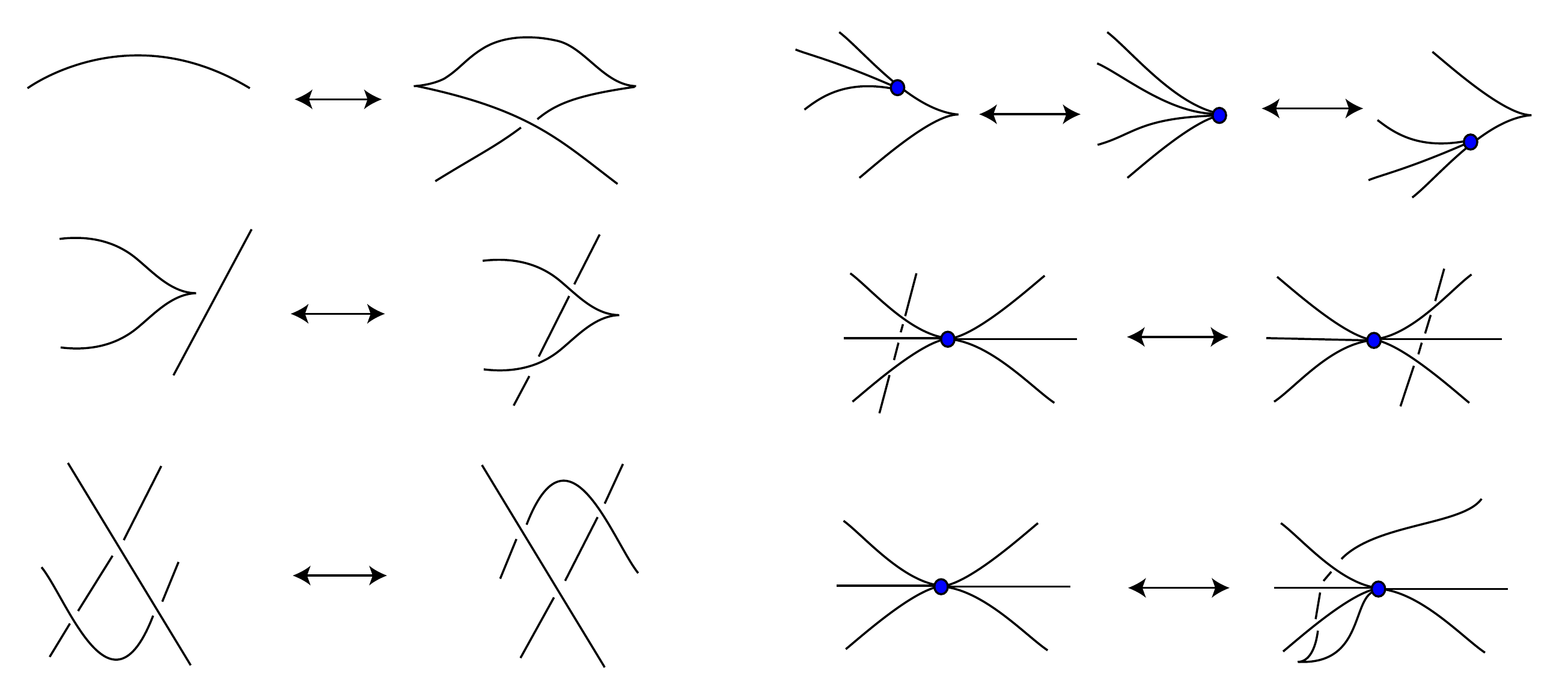}}
\put(84, 153){I}
\put(82,99){II}
\put(80,33){III}
\put(256, 150){IV}
\put(328, 151){IV}
\put(295,93){V}
\put(291,30){VI}
\end{picture}
\caption{\small Legendrian isotopy moves for graphs:  Reidemeister I, II and III moves, a vertex passing through a cusp (IV), an edge passing under or over a vertex (V), an edge adjacent to a vertex rotates to the other side of the vertex (VI). Reflections of these moves that are Legendrian front projections are also allowed.}\label{moves}
\end{center}
\end{figure}

\begin{remark} Through Legendrian isotopy, the order of the edges around a vertex does not change.
Since two edges which share a vertex can never have the same (oriented) tangent vector throughout a Legendrian isotopy, when seen in the contact plane at the vertex, the order of the edges around the vertex is the same up to a cyclic permutation.
\end{remark}

%\marginpar{\tiny \hspace*{0.4in} proof of remark}
We extend the classical invariants  $tb$ and $rot$ to Legendrian graphs.  
A cycle in a Legendrian graph is a \emph{piecewise smooth Legendrian knot}, that is, a simple closed curve that is everywhere tangent to the contact planes, but has finitely many points, at the vertices, where it may not be smooth. 
We first define the invariants, $tb$ and $rot$, for piecewise smooth Legendrian knots and then we extend the definition to Legendrian graphs.

Given any piecewise smooth Legendrian knot $K$, there is a natural way to construct a smooth knot from $K$. We define the \textit{standard smoothing} of a cycle $K$  of a Legendrian graph to be the $C^1-$curve obtained by replacing $K$ in an $\epsilon-$neighborhood of each vertex by the lift  of a minimal front projection of a smooth curve, which coincides with $K$ outside of the $\epsilon-$neighborhood.  
The projection is minimal, in the sense that no extra stabilizations or knotting are introduced with this approximation. See Figure \ref{smoothing}. 
We denote the standard smoothing of a piecewise smooth Legendrian knot $K$ by $K_{st}$.  \\

\begin{proposition}\label{isosmoothing}
If $K_1$ and $K_2$ are piecewise smooth Legendrian knots which are isotopic as Legendrian graphs, then their standard smoothings are isotopic Legendrian knots.
\end{proposition}
\begin{proof}
Since $K_1$ and $K_2$ are isotopic as Legendrian graphs, their front projects, $\tilde{K_1}$ and $\tilde{K_2}$, are related by a finite sequence of the standard Reidemeister moves I, II, III, as well as, the moves IV, V, VI, and Legendrian planar isotopy.  Thus to show that the standard smoothings of $K_1$ and $K_2$ are isotopic, we need only show that the standard smoothings of valence two subgraphs of the Reidemeister moves IV, V, and VI are moves that result in isotopic Legendrian knots.  

\underline{Reidemeister move IV:}  For Reidemeister move IV there are two different possibilities of valence two subgraphs up to reflection and planar isotopy to be considered, see Figure \ref{moveIVI}(a).  
After smoothing the move either shows no changed, or a minor planar isotopy.

\underline{Reidemeister move V:}  For Reidemeister move V there are two different possibilities of valence two subgraphs up to reflection and planar isotopy to be considered, see Figure \ref{moveIVI}(b).  
After smoothing the move either shows a minor planar isotopy, or a difference of a Reidemeister II move.

\underline{Reidemeister move VI:}  For Reidemeister move VI there are three different possibilities of valence two subgraphs up to reflection and planar isotopy to be considered, see Figure \ref{moveIVI}(c).  
After smoothing the move either shows no change, a minor planar isotopy, or a difference of a Reidemeister I move. 

In all cases the standard smoothing can be moved to the other standard smoothing by a single Reidemeister move or planar isotopy.  
\end{proof}

%---- Figure ---- smoothing -----------

\begin{figure}[htpb!]
\begin{center}
\begin{picture}(360, 100)
\put(0,0){\includegraphics{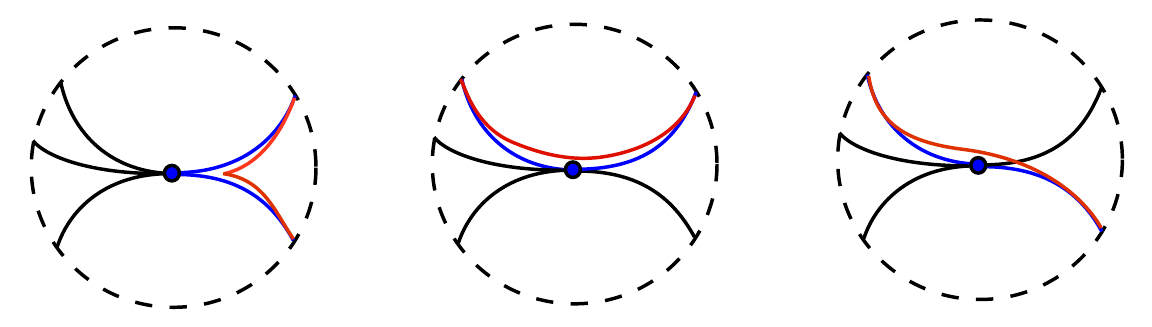}}
\end{picture}
\caption{\small Approximation by a $C^1-$curve near a vertex. The blue edges which are part of the cycle $K$ are replaced by the red $C^1-$arc near the vertex. }\label{smoothing}
\end{center}
\end{figure}

%---- Figure ---- Smoothing moveIVI-----------

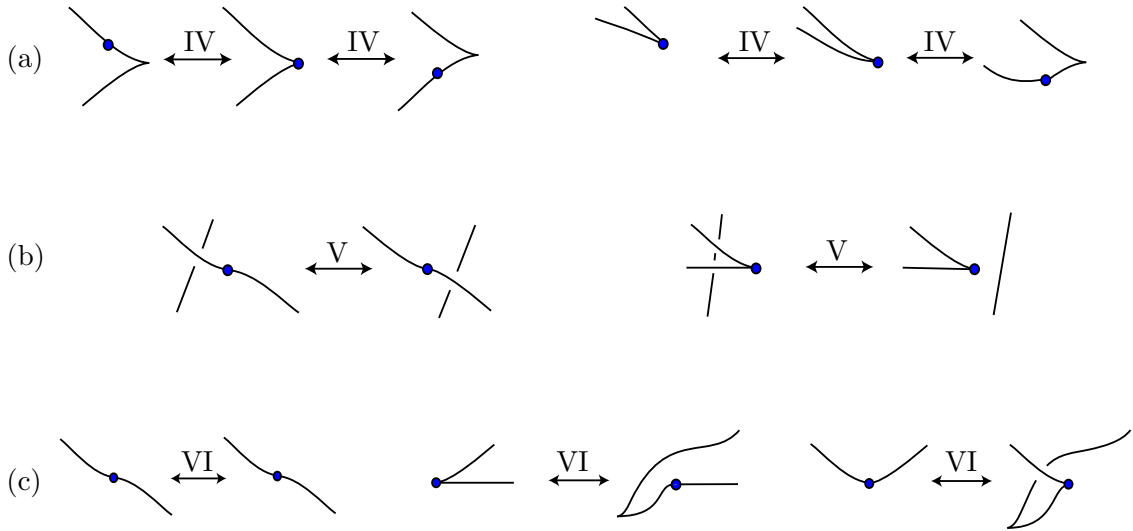
\begin{figure}[htpb!]
\begin{center}
\begin{picture}(410, 216)
\put(10,0){\includegraphics{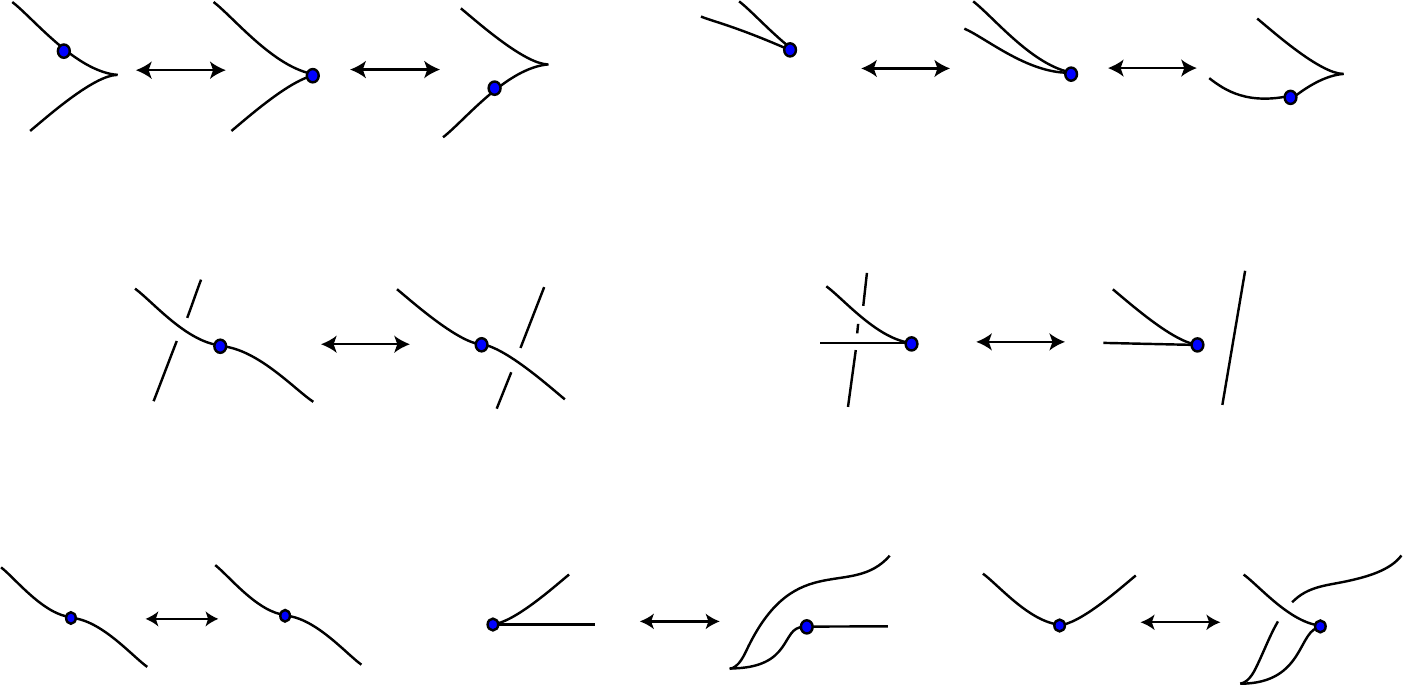}}
\put(-10,15){(c)}
\put(-10,100){(b)}
\put(-10,175){(a)}
\put(57,181){IV}
\put(119,181){IV}
\put(267,181){IV}
\put(337,181){IV}
\put(111, 101){V}
\put(300, 102){V}
\put(56,22){VI}
\put(198,22){VI}
\put(345,22){VI}
\end{picture}
\caption{\small Subdiagrams of Legendrian graphs moves.}\label{moveIVI}
\end{center}
\end{figure}

Of course, for any piecewise smooth Legendrian knot there are isotopic Legendrian knots, which can be obtained by a small isotopy near each vertex moving the edges so that they have parallel tangents at the vertex.  
In light of this, Proposition \ref{isosmoothing} shows that the isotopy class of a piecewise smooth Legendrian knot contains Legendrian embeddings from exactly one isotopy class of a Legendrian knot, the knot obtained by its standard smoothing.  
So given a piecewise smooth Legendrian knot $K$, and its standard smoothing $K_{st},$ one could take the definition of the classical invariants to be: $tb(K)=tb(K_{st})$ and $rot(K)=rot(K_{st})$. 
However, since we will be using the said definitions to define invariants for Legendrian graphs, we would like to defined them in such a way that smoothings are not needed. 
It should be noted that any other definition of $tb$ and $rot$ that coincides with the invariants for smooth Legendrian knots will be equivalent to the above definition, as a result of Proposition \ref{isosmoothing}.

% -------subsection The Thurston-Bennequin number------

\subsection{ The Thurston-Bennequin number.}
 Let $K$ represent a piecewise smooth Legendrian knot and $v$ a vector field along $K$ which is transverse to the contact planes. Take the push-off of $K$, $K'$ in the direction of $v$ and let $tb(K):=lk(K,K')$. This definition coincides with that for smooth knots. There is no obstruction given by the direction change at the vertices.

\begin{definition}
For a Legendrian graph $G$,  we fix an order on the cycles of $G$ and we define \textit{the Thurston-Bennequin number of $G$}, denoted by $tb(G)$, to be the ordered list of the Thurston-Bennequin numbers of the cycles of $G$.
If $G$ has no cycles, we define $tb(G)$ to be the empty list.
\end{definition}

\subsection{ The rotation number.}

 We define the rotation number of a null-homologous piecewise smooth Legendrian knot $K\subset (\R^3, \xi_{std})$ as follows:
 Consider $\Sigma$ with $\partial\Sigma= K$ and endow $K$ with the orientation induced by that on $\Sigma$.
 Consider the trivialization of  $\xi_{std}$ given by the two vectors $d_1=\frac{\partial}{\partial y}$ and $d_2=-y\frac{\partial }{\partial z}-\frac{\partial}{\partial x}  $.
 Denote by $p_1$, $p_2$,..., $p_s$, $p_{s+1}=p_1$ the vertices on $K$, in cyclic order as given by the orientation, denote by $e_i$, $i=1,...,s$, the smooth edge of $K$ between $p_i$ and $p_{i+1}$, and denote by $v_i$, $i=1,...,s$, the unit vector field tangent to $e_i$ pointing in the direction of the orientation on $K$.
We follow the $v_i's$ in the trivialization given by $d_1$ and $d_2$ and count the number of times $d_1$ is passed, with sign.
At each vertex $p_i$, $i=2,...,s+1$ if $v_{i-1}$ and $v_i$ do not coincide we complete the rotation counterclockwise if $\{v_{i-1}, v_i\}$ is a positively oriented basis for $\xi_{p_i}$ and clockwise if $\{v_{i-1}, v_i\}$ is a negatively oriented basis for $\xi_{p_i}$. This is equivalent to completing by a rotation from $v_{i-1}$ towards $v_i$ in the direction of the shortest angle between the two. Note that, since one edge is oriented towards the vertex and one edge is oriented away from the vertex, $v_{i-1}$ and $v_i$ cannot be opposite to each other.

 Denote by $p(K)$ the number of times $\pm d_1$ is passed in the counterclockwise direction and by $n(K)$ the number of times $\pm d_1$ is passed in the clockwise direction as $K$ is traced once. We define $rot(K)$ by
 $$rot(K):= \frac{1}{2}(p(K)-n(K)).$$
 
 Let $\mathcal{L}(K)$ denote the Legendrian isotopy class of $K$.
 The above discussion gives a recipe of how to compute the rotation number for a particular embedding of a piecewise smooth Legendrian knot.
However, when $K$ changes through Legendrian isotopy, the tangent vector at $K$ changes continuously and the edges cannot pass over one another at the vertices.  Thus we get a continuous map $rot: \mathcal{L}(K)\rightarrow \Z$. 
The rotation number is therefore a Legendrian isotopy invariant for piecewise smooth Legendrian knots.
If $K$ is a smooth Legendrian knot then we recover the rotation number for $K$.  

\begin{definition}
For a Legendrian graph $G$,  we fix an order on the cycles of $G$ with orientation and we define \textit{the rotation number of $G$}, denoted by $rot(G)$, to be the ordered list of the rotation numbers of the cycles of $G$.  
If $G$ has no cycles, we define $rot(G)$ to be the empty list.
\end{definition}

%-----------section about Non-realization of max tb------

\section{Non-realization of maximal $tb$}\label{maxtb}

In \cite{M}, Mohnke proved that the Borromean rings and the Whithead link cannot be represented by Legendrian links of trivial unknots.  A \emph{trivial unknot} is one with $tb=-1$ and $rot=0$ (like the first unknot shown in Figure \ref{unknot_trefoil}). The trivial unknot attains the maximal possible Thurston-Bennequin number for an unknot.  The obstruction comes from upper bounds on $tb$ given by the minimal degree of one of the variables in the Kauffman polynomial.  

In this section we determine which graphs can be Legendrian realized in such a way that all cycles are trivial unknots. 
We will see that there are many graphs (even planar graphs) with no such Legendrian realization.
We give a full characterization of these graphs. 
We also present a more general result about which graphs can be realized with all their cycles having maximal $tb$ for a class of knots. 
The following two lemmas will be useful.

%---- Figure ----

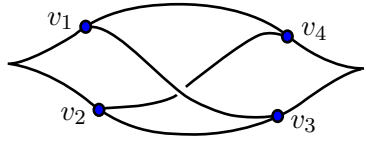
\begin{figure}[htpb!]
\begin{center}
\begin{picture}(144, 50)
\put(0,0){\includegraphics{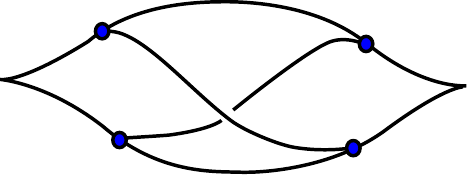}}
\put(15,42){$v_1$}
\put(20,6){$v_2$}
\put(107,3){$v_3$}
\put(111,38){$v_4$}
\end{picture}
\caption{ \small A Legendrian embedding of $K_4$.}\label{K4}
\end{center}
\end{figure}

\begin{lemma}
The front projections of  any two Legendrian realizations of a graph in $(\R^3, \xi_{std})$ are related by a finite sequence of Legendrian isotopies and changes involving
\begin{enumerate}
\item[(a)] one edge (i.e. stabilizations, crossings of the edge with itself)
\item [(b)] two adjacent edges (i.e. change in the number of crossings between the two edges)
\item [(c)] two nonadjacent edges (i.e change in the number of crossings between the two edges)
\end{enumerate}
\label{two_fronts}
\end{lemma}
\begin{proof}
If the two Legendrian realizations are Legendrian isotopic then the two front projections differ by Legendrian isotopy only. If not, the two front projections may differ in number of edge stabilizations or number of crossings. The stabilizations occur on a single edge while the crossings may occur on a single edge, between two adjacent edges or between two nonadjacent edges.  
\end{proof}

\begin{lemma} For any Legendrian embedding $L$ of $K_4$ in $(\mathbb{R}^3, \xi_{std})$
$$\displaystyle\sum_{\gamma \in \Gamma_L}tb(\gamma) \equiv 0  \bmod 2, $$
where $\Gamma_L$ is the set of cycles in $L$.  
\label{mod2}
\end{lemma}

\begin{proof}
Denote by $v_1$, $v_2$, $v_3$ and $v_4$ the vertices of an abstract $K_4$.  There are seven cycles in $K_4$: four 3-cycles  ($v_1v_2v_3, v_1v_2v_4, v_1v_3v_4, v_2v_3v_4)$ and three 4-cycles $(v_1v_2v_3v_4, v_1v_2v_4v_3,$\\ $ v_1v_3v_2v_4).$ Each edge appears in four different cycles and each pair of edges appears in two different cycles.
Consider the embedding $K$ of $K_4$ shown in Figure \ref{K4}. For this embedding, there are six cycles with $tb=-1$ and one cycle ($v_1v_2v_4v_3$) with $tb=-2$, thus
$$\displaystyle S=\sum_{\gamma \in \Gamma_K} tb(\gamma)=-8. $$
Take an arbitrary Legendrian embedding of $K_4$, call it $L$. By Lemma \ref{two_fronts},  the front projection for this embedding differs from  $K$ by a finite sequence of Legendrian isotopy and changes of the form described above in one edge, two adjacent edges, or two non-adjacent edges.
First, since $tb$ is an invariant, $S$ does not change under isotopy. 
\begin{enumerate}
\item If a change in a single edge is made, the $tb$ for all four cycles containing this edge is modified by the same quantity.
\item If a change in two adjacent edges is made, the $tb$ for both cycles containing this pair of edges is modified by the same quantity.
\item If a change in two non-adjacent edges is made, the $tb$ for both cycles containing this pair of edges is modified by the same quantity.
\end{enumerate}
Thus, the parity of the sum of the $tb$'s over all cycles remains unchanged throughout the process and $$ \displaystyle\sum_{\gamma \in \Gamma_L} tb(\gamma) \equiv 0  \bmod 2. $$

\end{proof}
We define $L_{odd} $ to be the set of topological knot classes with odd maximal Thurston-Bennequin number. 
 We have the following theorem:

\begin{theorem}
Let $G$  be a graph that contains $K_4$ as a minor. Then there does not exist a Legendrian realization of $G$ such that all of its cycles are knots in $L_{odd} $ realizing their maximal Thurston-Bennequin number. \label{thmG}

\end{theorem}
\begin{proof}
It suffices to prove the theorem for $G=K_4$, since any graph that contains $K_4$ as a minor contains a subdivision of $K_4$ (see Definition \ref{sub}).
 Assume all seven cycles of $K_4$  can be realized with maximal odd Thurston-Bennequin number, $2t_n+1, n=1, \dots,7$.  Then,
$ \displaystyle\sum_{n=1}^{7} (2t_n+1) \not \equiv 0 \bmod 2$, contradicting the conclusion of Lemma \ref{mod2}.
 \end{proof}
 
Since the unknot has odd maximal $tb=-1$ we obtain the following corollary.
Recall that the trivial unknot is the unknot with maximal Thurston-Bennequin number.
 
\begin{corollary}
Let $G$  be a graph that contains $K_4$ as a minor. Then there does not exist a Legendrian realization of $G$ such that all of its cycles are trivial unknots. \label{unknot}
\end{corollary}

When we consider a graph $G$ that contains $K_4$ as a minor and an arbitrary Legendrian embedding $f(G),$ Corollary~\ref{unknot} guaranties, either $f(G)$ contains a nontrivial knot, or that all of the cycles of $f(G)$ are unknots, but they do not all have maximal $tb$.  
For any intrinsically knotted graph (see Section \ref{spatialgrph}) Corollary \ref{unknot} is not surprising, since 
every embedding of an intrinsically knotted graph contains a nontrivial knot.  While there are many graphs 
that contain $K_4$ as a minor and are not intrinsically knotted, all of the intrinsically knotted graphs contain $K_4$ as a minor.  
This can be seen, by noting that $K_4$ is a minor of all graphs in the Petersen family.  Then since all intrinsically knotted 
graphs are also intrinsically linked, we see $K_4$ is a minor of all of the intrinsically knotted graphs.

The same should be considered for Theorem~\ref{thmG}.  Given an arbitrary Legendrian 
embedding $f(G),$ of a graph $G$ that contains $K_4$ as minor, Theorem~\ref{thmG} implies 
that either $f(G)$ does not only contain knots from $L_{odd},$ or $f(G)$ contains only knots 
from $L_{odd},$ and they do not all attain their maximal $tb$.  So Theorem~\ref{thmG} gives 
more information about intrinsically knotted graphs.  One example of an intrinsically knotted graph is $K_7$, (See  \cite{CG}). 
There are embeddings of $K_7$ with unknots as all but one cycle which is a trefoil.  Both the unknot 
and the right handed trefoil have odd maximal Thurston-Bennequin number, so then 
Theorem~\ref{thmG} implies for such a Legendrian embedding not all of the cycles will attain their maximal $tb$.  
There are many other knots whose maximal Thurston-Bennequin number is odd. 
A few examples are the right handed trefoil (Figure \ref{unknot_trefoil}), the figure eight knot, the $6_1$ knot 
and its mirror image, the $6_2$ knot and its mirror image.

\begin{remark}There is a more general version of Theorem~\ref{thmG} that is an immediate consequence of the proof, 
though it is cumbersome to state.  Let $G$  be a graph that contains $K_4$ as a minor. Let $S$ be the subdivision of 
$K_4$ that $G$ contains.  Then there does not exist a Legendrian realization of $G$ such that $S$ contains precisely 
an odd number of cycles that are knots in $L_{odd},$ where all the knots in $S$ realize their maximal Thurston-Bennequin number.  
\end{remark}

If we focus on embeddings with only unknots as in Corollary \ref{unknot} the converse also holds.
Before proving this we introduce some needed definitions and observations. 

\begin{definition}
\begin{enumerate}
\item[]
\item A \textit{path} between two vertices $v_1$ and $v_2$ of a graph $G$ is a finite sequence of at least two edges starting at $v_1$ and ending at $v_2$, with no repetition of vertices.
\item A vertex of the graph $G$ is said to be \textit{a cut vertex} if by deleting the vertex (and all incident edges) the resulting graph has more connected components than $G$. 
\item An edge of $G$ is said to be a \textit{cut edge} if by deleting the edge the resulting graph has more connected components than $G$. 
\item A \textit{subdivision} of the graph $G$ is a graph obtained by replacing a finite number of edges of $G$ with paths (one can  think of this as adding a finite number of vertices along edges of $G$).
\end{enumerate}
\label{sub}
\end{definition}

\begin{remark}Let $G$ be a graph which does not contain $K_4$ as a minor.  Let $v_1$, $v_2$, $v_3$, $v_4$ be four vertices in a cycle of $G$, appearing in this order. For any such formation,  there are not two edges or paths, other than the ones already contained in the cycle, connecting $v_1$ and $v_3$ and connecting $v_2$ and $v_4$. Otherwise $v_1$,  $v_2$, $v_3$, and $v_4$ represent the vertices of a (subdivision of) $K_4$. See Figure \ref{noK4_remarks}(a). \label{alternates}
\end{remark}

\begin{remark}Let $G$ be a graph which does not contain $K_4$ as a minor.  Let $v_1$, $v_2$, $v_3$, be three vertices in a cycle of $G$. For any such formation,  there is not an additional vertex $v$ with distinct edges or paths, connecting $v$ to the other vertices $v_1$, $v_2$, and $v_3$. Otherwise $v_1$,  $v_2$, $v_3$, and $v$ represent the vertices of a (subdivision of) $K_4$. See Figure \ref{noK4_remarks}(b). \label{triangle}
\end{remark}

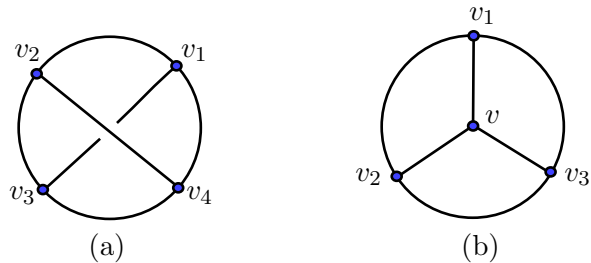
\begin{figure}[htpb!]
\begin{center}
\begin{picture}(200, 70)
\put(0,0){\includegraphics{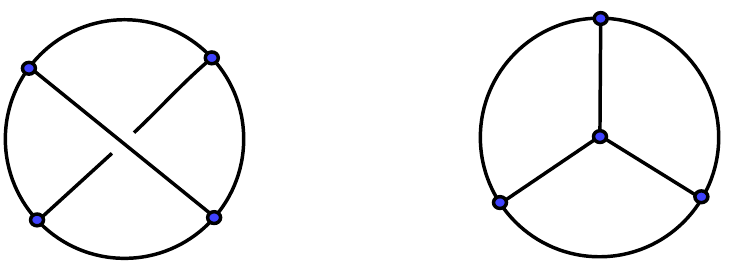}}
\put(28,-12){(a)}
\put(169,-12){(b)}
\put(63,63){$\small {v_1}$}
\put(0,63){$\small {v_2}$}
\put(-2,9){$\small {v_3}$}
\put(65,9){$\small {v_4}$}
\put(178,38){$\small {v}$}
\put(172,77){$\small {v_1}$}
\put(129,16){$\small {v_2}$}
\put(208,16){$\small {v_3}$}
\end{picture}
\caption{\small Embeddings of the graph $K_4$ as described in Remarks \ref{alternates} and \ref{triangle}.}\label{noK4_remarks}
\end{center}
\end{figure}

\begin{theorem}
Let $G$  be a graph, if $G$ does not contain $K_4$ as a minor then it can be Legendrian realized in  $(\R^3, \xi_{std})$ in such a way that all its cycles are trivial unknots. \label{thmGrec}

\end{theorem}

\begin{proof}
We need only prove the theorem for $G$ connected, with no cut edges and no cut vertices. In all other cases such components (connected components with no cut edges nor cut vertices) of $G$ can be realized in the same way and cut edges can be realized in any fashion (as they do not appear in any cycle).

In what follows, all edges of $G$ are realized as non-stabilized arcs. Let $C$ be one of the cycles of $G$ and $s$ the number of vertices of $C$. 

Realize $C$ in such a way that its front projection consists of $s-1$ horizontal edges and one edge on top of these. Label the vertices on $C$ with $v_1$, $v_2$, ... , $v_s$, in this order, from left to right. 
Realize all the other edges between vertices of $C$ on top of the horizontal edges of $C$, in a nested fashion as shown in Figure \ref{noK4_one}(a).  
By Remark \ref{alternates}, we can realize these edges without any crossings. Otherwise, the four endpoints of two crossing edges represent the vertices of a $K_4$.

Next, for each pair of vertices $(v_i, v_j)$, $i<j$, of $C$ for which there exists at least one path between $v_i$ and $v_j$  not containing any edges of $C$, realize one of these paths under the horizontal line of $C$.  
Call this path $P_{ij}^1$. For each $t >1$, if there is another path between $v_i$ and $v_j$ not containing edges of $C$ or edges of $P_{ij}^1$, ..., $P_{ij}^{t-1}$, realize this path under $P_{ij}^{t-1}$.  See Figure \ref{noK4_one}(b). Denote by 
$$\mathcal{P}_{ij}:= \{P_{ij}^1, P_{ij}^2, ... , P_{ij}^t, ... \}$$
This is a finite set.

By construction, all paths in $\mathcal{P}_{ij}$ are nested, without any crossings, they have the vertices $v_i$ and $v_j$ in common and are disjoint otherwise.
By Remark \ref{alternates}, we can realize all the paths appearing thus far without any crossings.

Additionally, for any vertices $v_i$, $v_j$, $v_k$ and $v_l$ of $C$, with $(v_i,v_j)\ne (v_k,v_l)$,  no element of $\cp_{ij}$ has any vertices (distinct from $v_i$, $v_j$, $v_k$ and $v_l$) or edges in common with any element of $\cp_{kl}$.  
We prove the observation by contradiction.  
There are four different cases of $(v_i,v_j)\ne (v_k,v_l)$, that is: (1) $i=k$ and $j\ne l,$ (2) $i\ne k$ and $j=l,$ (3) $j=k$, and (4) $i, j, k,$ and $l$ are all distinct.
Suppose a path $P$ in $\cp_{ij}$ and a path $Q$ in $\cp_{kl}$ have at least one vertex other than $v_i$, $v_j$, $v_k$ and $v_l$ in common.  
Let $\mathcal{V}$ be the non-empty set of vertices that the paths $P$ and $Q$ have in common distinct from $v_i$, $v_j$, $v_k$ and $v_l$.
Let $w_1$ be the left most vertex on $P$ with $w_1\in\mathcal{V}$, and let $w_2$ be the right most vertex on $P$ with $w_2\in\mathcal{V}$.  (We need only show that a contiguous overlap cannot occur, for if there are paths with multiple overlaps there are paths with a single overlap.)   For the subcases (1), (3), and (4), let $w_i\in\{w_1, w_2\}$ be the vertex in this set closest to $v_l$ on the path $Q$.  
Then the part of the path $Q$ from $w_i$ to $v_l$ is distinct from the path $P$.  
So this path between $w_i$ and $v_l$ together with the two parts of $P$ going from $v_i$ to $w_i$ and from $w_i$ to $v_j$ are all distinct paths joining $w_i$ to the cycle $C$.  
Thus by Remark \ref{triangle} we have a contradiction.  Similarly, for the case (3)  let $w_i\in\{w_1, w_2\}$ be the vertex in this set closest to $v_i$ on the path $P$.  
Then the part of the path $P$ from $w_i$ to $v_i$ together with the two parts of $Q$ going from $v_k$ to $w_i$ and from $w_i$ to $v_l$ are all distinct paths joining $w_i$ to the cycle $C$.  
Thus by Remark \ref{triangle} we have a contradiction.

Further, we treat each of the elements in $\cp_{ij}$, $\forall i,j$,   $1\le i<j \le s$, as if it were the path of horizontal edges in $C$. 
We realize the edges with vertices on each such path above the path and the paths with endpoints on each such path below the path.  
We continue to realize each new edge on top of the path its vertices end lie on and each new path below the path its vertices lie on.

%-------------- Figure -----------------

\begin{figure}[htpb!]
\begin{center}
\begin{picture}(450, 85)
\put(0,0){\includegraphics{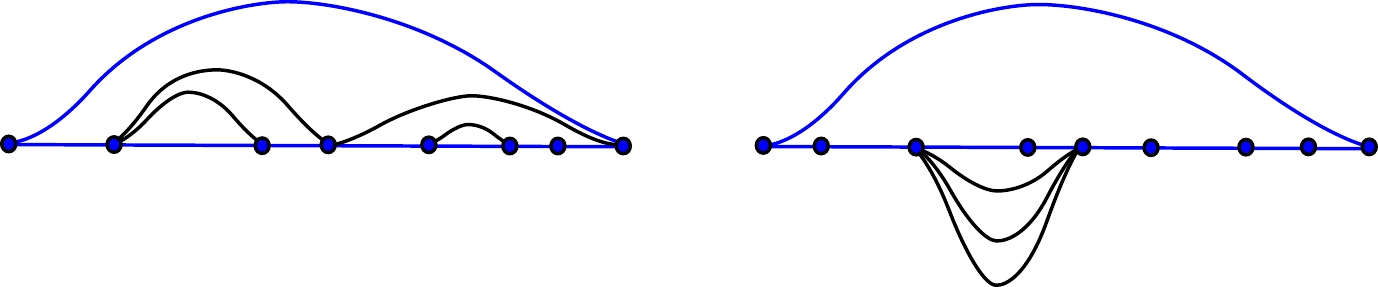}}
\put(85,0){(a)}
\put(305,0){(b)}
\put(60,85){\color{blue}{$C$}}
\put(288,87){\color{blue}{$C$}}
\put(0,32){$\small {v_1}$}
\put(28,32){$\small {v_2}$}
\put(263,49){$\small {v_i}$}
\put(310,49){$\small {v_j}$}
\put(160,32){$\small {v_{\tiny s-1}}$}
\put(180,32){$\small {v_s}$}
\end{picture}
\caption{\small First steps in Legendrian realizing $G$ with all cycles unknots of maximal $tb$. The cycle $C$ is blue.}\label{noK4_one}
\end{center}
\end{figure}
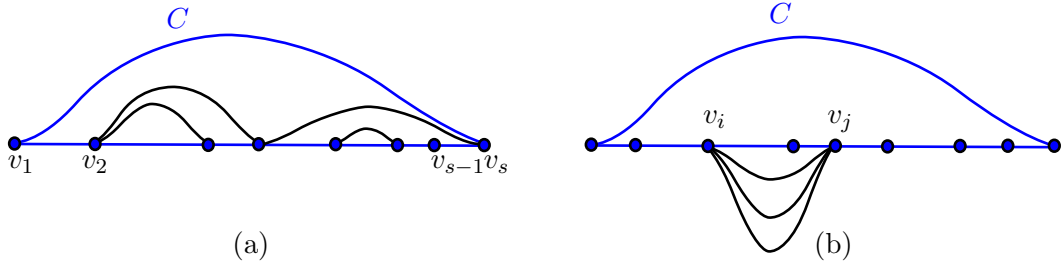

%-------- Figure--------

Claim:
\begin{enumerate}
\item  This is a  Legendrian embedding of $G$ (all edges and vertices of $G$ are realized).
\item  All cycles in this embedding have maximal $tb=-1$. 
\end{enumerate}
Proof of (1): The above construction is a prescription for a front projection thus as long as all vertices and edges of $G$ are realized we have a Legendrian embedding. 
The two cases of assuming a vertex was not realized and assuming an edge was not realized will be treated concurrently.  If there exists an edge of $G$ which has not been realized, then there is an edge or a  path in $G$ between two realized vertices that has not been realized.  
If there exists a vertex of $G$ which has not been realized, then since $G$ was assumed without cut edges and connected, there is a path in $G$ between two realized vertices that has not been realized.
We show that the existence of such an edge or path leads to a contradiction. 

Let $v$ and $w$ represent two realized vertices of $G$ such that there is a non-realized edge or path (containing exclusively edges that have not been realized) between $v$ and $w$. 
The vertices $v$ and $w$ cannot both be in $C$, since all paths and edges between these vertices have been realized. 
For the same reason, $v$ and $w$ cannot both be vertices on the same element of a $\cp_{ij}$, or on the same path realized at a later stage.  

From the vertex $w$ we will form a cycle, $\mathcal{C}_w$.  
The cycle $\mathcal{C}_w$ is made by taking the path to the right of $w$ formed by choosing the edge that is to the right and upper most at each vertex (this path will connect $w$ with $v_s$), the path to the left of $w$ made by taking the edge that is to the left and upper most at each vertex (this path will connect $w$ with $v_1$), together with the edge connecting $v_1$ and $v_s$.  
The vertex $v$ is not on the cycle $\mathcal{C}_w$ , or we would have realized a path between these vertices.  
Now let $\cp_x$ be the path formed by starting at $v$ choosing the edge that is to the right and upper most at each vertex until the vertex is a vertex of $\mathcal{C}_w$ , call this vertex $x$, and let $\cp_y$ be the path formed by starting at $v$ choosing the edge that is to the left and upper most at each vertex until the vertex is a vertex of $\mathcal{C}_w$, call this vertex $y$.  
The vertices $w$, $x$ and $y$ are all on the cycle $\mathcal{C}_w$, the (edges or) paths $\cp_x$, $\cp_y$, and the unrealized edge or path from $v$ to $w$ are all distinct, thus by Remark \ref{triangle}, this cannot occur.    
Therefore there is no such unrealized edge or path.  

 Proof of (2): Since the construction does not contain any crossings, we need only show that no cycle exhibits a stabilization. 
Assume there is a cycle in the embedding which represents a stabilized unknot. 
Since all edges were realized as non-stabilized arcs in the first place, at each cusp of this unknot there is a vertex of $G$.  
Since a stabilization occurs, there are at least two left cusps and two right cusps.  
We denote the vertices at the left cusps by $w_1$ and $w_3$ and the vertices at the right cusps by $w_2$ and $w_4$, with the four vertices appearing in order $w_1$, $w_2$, $w_3$, $w_4$ in the cycle. 
Without loss of generality we may assume that either $w_2$, $w_3$, and $w_4$ are at consecutive right-left-right cusps and $w_1$ is on the left of $w_3$ (Figure \ref{noK4_three}, (a)) or $w_1$, $w_4$, and $w_3$ are consecutive left-right-left cusps and $w_2$ is on the right of $w_4$ (Figure \ref{noK4_three}, (b)). We show that a $K_4$ necessarily exists.

%-------- Figure--------
\begin{figure}[htpb!]
\begin{center}
\begin{picture}(380, 108)
\put(0,0){\includegraphics{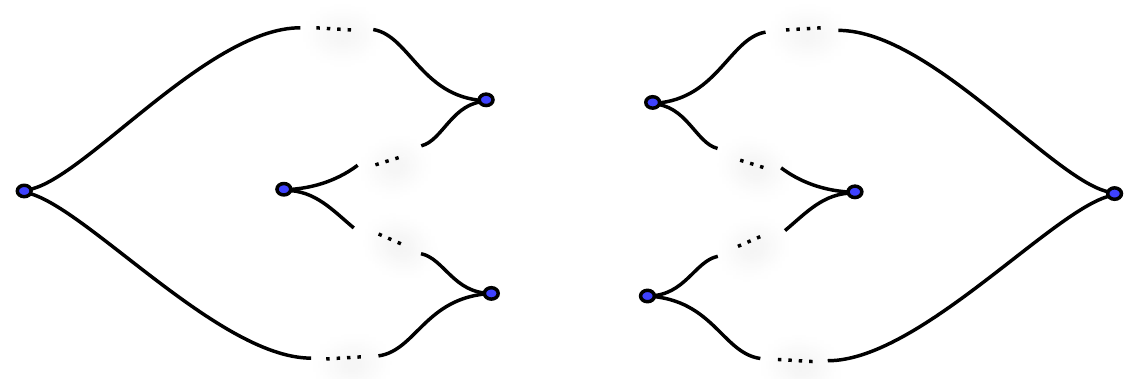}}
\put(-8,52){$w_1$}
\put(144,78){$w_2$}
\put(66,52){$w_3$}
\put(146,22){$w_4$}
\put(174,78){$w_1$}
\put(326,51){$w_2$}
\put(172,22){$w_3$}
\put(251,52){$w_4$}
\put(78,-10){(a)}
\put(248,-10){(b)}
\end{picture}
\caption{\small No cycle is a stabilized unknot.}\label{noK4_three}
\end{center}
\end{figure}
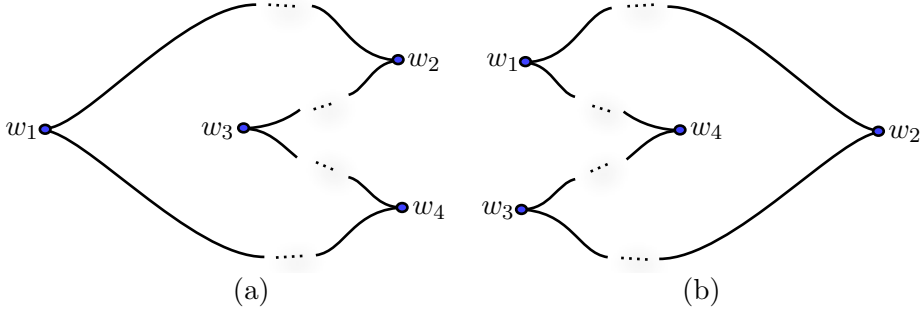

We prove the case pictured in Figure \ref{noK4_three}(a). 
The other case is similar.  
We consider a path  $P_2$ which starts at $w_2$ and always follows along an edge to the right at each vertex. 
We also consider a path $P_4$ which starts at $w_4$ and always follows along an edge to the right at each vertex. 
Denote the first vertex where $P_2$ and $P_4$ intersect by $u$. 
The vertex $u$ may be $w_2$, $w_4$, $v_s$, or a point in between. 
The path which follows $P_2$ from $w_2$ to $u$ and then $P_4$ from $u$ to $w_4$ is a path between $w_2$ and $w_4$, which is disjoint from both $w_1$ and $w_3$, since both $w_1$ and $w_3$ lie on the left of both $w_2$ and $w_4$. 
This is a third path between $w_2$ and $w_4$, in addition to the two included in the stabilized cycle. Now, consider a path $P_3$ which starts at $w_3$ and always follows along an edge to the left at each vertex.  
The path $P_3$ will eventually reach $v_1$, and will intersect the stabilized cycle in a vertex $u'$ situated on the arc of the cycle which goes between $w_2$ and $w_4$ and does not contain $w_3$.  
The vertices $u'$, $w_2$, $w_3$, $w_4$ are the vertices of a (subdivision of) $K_4$.  
\end{proof}

\begin{remark} Recall, a graph $G$ is \emph{minor minimal} with respect to a property if $G$ has the property, but no minor of $G$ has the property.  
Corollary \ref{unknot} together with Theorem \ref{thmGrec} show that $K_4$ is minor minimal with respect to the property of not having a Legendrian embedding with all cycles trivial unknots. 
Not only that, but Theorem \ref{thmGrec} shows that $K_4$ is the only graph in this minor minimal set,  thus characterizing this property.  
\end{remark}

%----------------section-------------- Classification (tb, rot)

\section{Legendrian graphs classified by classical invariants}\label{Classification}

It is known that certain types of Legendrian knots and links are determined by the classical invariants $tb$ and $rot$ in $(\mathbb{R}^3, \xi_{std})$.  
In \cite{EF}, Eliashberg and Fraser showed that the Legendrian unknot is determined by $tb$ and $rot$. 
In \cite{EH}, Etnyre and Honda showed the same holds for torus knots and the figure eight knot, and in \cite{DG}, Ding and Geiges showed that links consisting of an unknot and a cable of that unknot, are classified by their oriented link type and the classical invariants in $(\mathbb{R}^3, \xi_{std})$.

In this section we investigate what types of spatial graphs are classified up to Legendrian isotopy by the pair $(tb, rot)$.
It is useful to recall that within each Legendrian isotopy class, at each vertex, the edges appear in a fixed cyclic order.  
For now we consider graphs that have cut edges or cut vertices.
\begin{remark}\label{cute}
Let $G$ be a graph containing a vertex of valence at least three which is incident to at least one cut edge. 
For different Legendrian realizations of $G$, the order of edges at this vertex can differ, while the classical invariants for all cycles are the same. 
This is because the cut edge $e$ does not appear in any cycle. 
See Figure \ref{tbrot}(a). 
That is, a Legendrian embedding of $G$ is not determined by the pair $(tb, rot)$. 

\end{remark}

%----- Figure----------------

\begin{figure}[htpb!]
\begin{center}
\begin{picture}(300, 150)
\put(0,0){\includegraphics{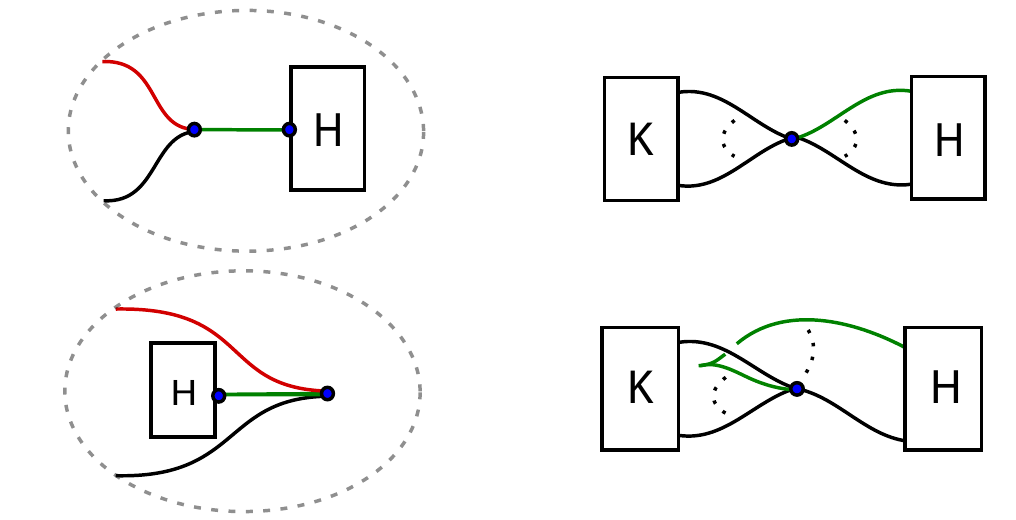}}
\put(65,-10){(a)}
\put(225, -10){(b)}
\put(68, 116){$e$}
\put(93, 28){$v$}
\put(52, 103){$v$}
\put(226, 100){$v$}
\put(227, 28){$v$}
\end{picture}
\caption{\small The order of edges around vertex $v$ is not the same above as below.}
\label{tbrot}
\end{center}
\end{figure}
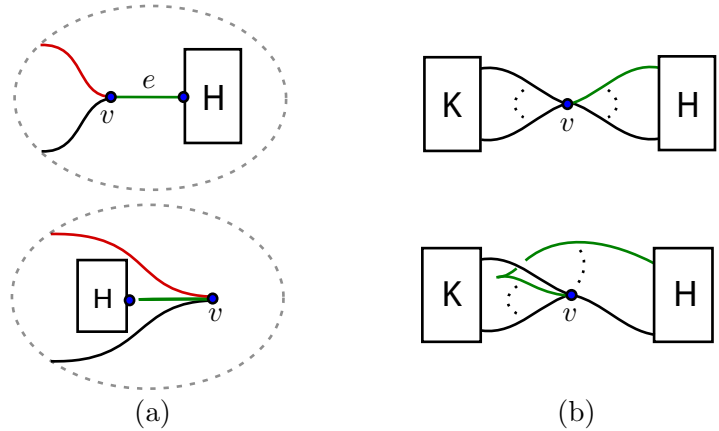

\begin{remark}\label{cutv}
Let $G$ be a graph containing a cut vertex of valence at least four. 
For different Legendrian realizations of $G$, the order of edges at this vertex can differ, while the classical invariants for all cycles are the same. 
For an example see Figure \ref{tbrot}(b). 
Note that there are no cycles containing edges from both $K$ and $H$. 
In particular, the modified edge is not  in such a cycle, thus $tb$ and $rot$ are the same for the two embeddings. 
That is, a Legendrian embedding of $G$ is not determined by the pair $(tb, rot)$. 
 \end{remark}

This means that even uncomplicated graphs carry more information as a whole than the set of knots represented by their cycles.
The above remarks can be summarized into the following:

\begin{proposition}
No graph containing at least one cycle and at least one cut edge or one cut vertex is determined up to Legendrian isotopy by the pair ($tb$,  $rot$).
\label{cutedge}
\end{proposition}
\begin{proof}  Suppose that $G$ has at least one cut edge or vertex and is determined by the pair ($tb$,  $rot$).  
Suppose $G$ has a cut vertex $v$, then by Remark \ref{cutv} the vertex $v$ must have valence 3 or less.    
If $v$ is a cut vertex with valence 3 then it must be adjacent to a cut edge, but by Remark \ref{cute} this would imply that $G$ is not determined by the pair ($tb$,  $rot$).  
So $v$ must have valence 2 or less.  However, to be a cut vertex it must have at least valence two.  
For a valence 2 cut vertex both incident edges are cut edges, so by Remark \ref{cute} both of the vertices defining these edges must have valence 2 or less.  
Therefore the only such graph with a cut vertex is a path graph.  
So $G$ has no cycles.  

Suppose $G$ has a cut edge $e$ with no cut vertices. By Remark \ref{cute} the valences of the two vertices must be 2 or less.
Since the two vertices cannot be cut vertices they are of valence one.
Thus the graph $G$ is a single edge.
Thus there does not exist a graph containing at least one cycle and at least one cut edge or one cut vertex that is determined by the pair ($tb$,  $rot$).
\end{proof}

Next, we focus on Legendrian embeddings of the lollipop graph and the handcuff graph. 
See Figure \ref{lollipop}.
Both these graphs have one cut edge. For any topological class of these graphs the Legendrian class cannot be determined by the pair $(tb, rot)$, by Proposition \ref{cutedge}.

\begin{figure}[htpb!]
\begin{center}
\begin{picture}(355, 48)
\put(0,0){\includegraphics{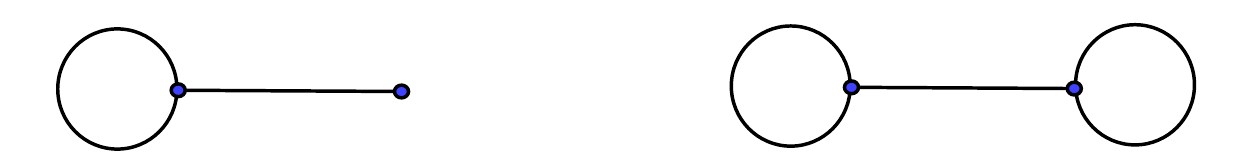}}
\put(80,-5){(a)}
\put(272, -5){(b)}
\end{picture}
\caption{\small (a) the lollipop graph  (b) the handcuff graph}
\label{lollipop}
\end{center}
\end{figure}

\begin{definition}
A \textit{planar spatial graph} is a spatial graph which is ambient isotopic to an embedding in the plane. 
A  \textit{ Legendrian planar graph} (or a \textit{planar Legendrian realization}) is a Legendrian realization of a planar spatial graph. 
\end{definition}

We show that if we restrict to planar spatial graphs, a pair $(tb, rot)$ determines exactly two Legendrian isotopy classes of the lollipop graph and
a pair $(tb, rot)$ determines exactly four Legendrian isotopy classes of the handcuff graph. 
We do this by constructing a Legendrian isotopy between an arbitrary embedding and a \textit{standard form} embedding. We define a standard form embedding below. 

\begin{definition}
\begin{enumerate}
\item[]
\item We say a Legendrian unknot is in \textit{standard form} if it is the lift of a front projection as in Figure \ref{standard_unknotEF}(a) or (b). The front projection in Figure ~\ref{standard_unknotEF}(a) represents two distinct Legendrian classes, depending on the chosen orientation. For the front projection shown in Figure ~\ref{standard_unknotEF}(b) both orientation give the same Legendrian class, we fix the orientation to be the one which makes the left cusp a down cusp.
\item We say a planar Legendrian realization of the lollipop graph is \textit{in standard form} if it is the lift of a front front projection consisting of one front projection of an unknot $U$ in standard form as in Figure \ref{standard_unknotEF} and a nonstabilized arc at the lower right cusp of the unknot. The arc can sit in one of two ways with respect to the other edge segments coming together at the vertex. 
We say the planar Legendrian  realization of the lollipop graph is in \textit{standard form A} or \textit{B} if the cut edge sits as in Figure \ref{standard_vertex_lollipop}(a) or (b), respectively.
\item We say a planar Legendrian realization of the handcuff graph is \textit{in standard form} if it is the lift of a front projection consisting of two non-crossing front projections of unknots $U_1$ and $U_2$ each in standard form as in Figure \ref{standard_unknotEF}, one on the left and one on the right, and a nonstabilized arc between the lower right cusp of the unknot on the left and the leftmost cusp of the unknot on the right. 
The arc can sit in one of two ways with respect to the other edge segments coming together at each vertex.
 We say the planar Legendrian realization of the handcuff graph is in \textit{standard form AA, AB, BA} or \textit{BB} if the cut edge sits as in Figure \ref{standard_vertex}(a), (b), (c) or (d), respectively.

\end{enumerate}
\end{definition}
Figure \ref{standard_handcuff} represents a handcuff graph in standard form $AA$, with both unknotted components with $rot \ne 0$.

%----- Figure -----------
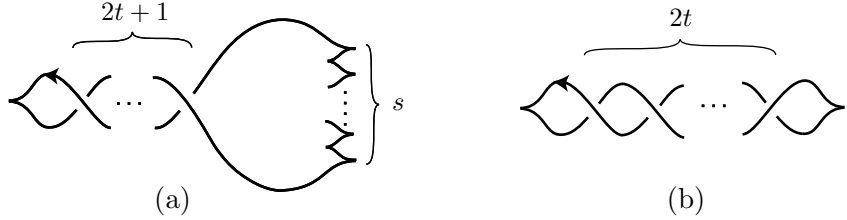
\begin{figure}[htpb!]
\begin{center}
\begin{picture}(330, 70)
\put(0,0){\includegraphics{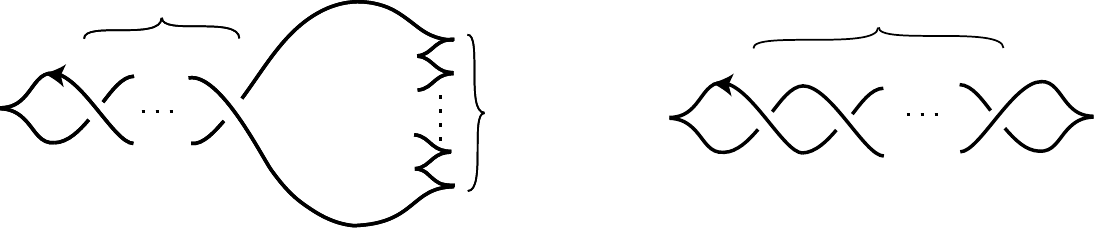}}
\put(55,-6){(a)}
\put(249, -6){(b)}
\put(35,65){\small $2t+1$}
\put(145, 31){\small $s$}
\put(250,63){\small $2t$}
\end{picture}
\caption{\small Legendrian unknot in standard form: (a) $rot (K) > 0$ (reverse orientation gives $rot(K)< 0$), (b) $rot(K) =0$.}
\label{standard_unknotEF}
\end{center}
\end{figure}

 %----- Figure -----------
\begin{figure}[htpb!]
\begin{center}
\begin{picture}(216, 32)
\put(0,0){\includegraphics{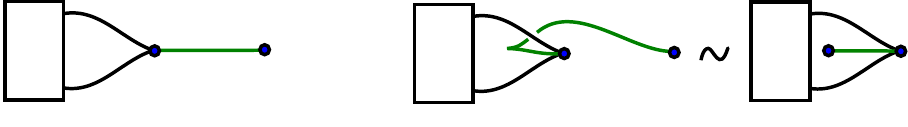}}
\put(30,-5){\small (a)}
\put(180,-5){\small (b)}
\put(7, 18){\small $U$ }
\put(125, 18){\small $U$ }
\put(222, 18){\small $U$ }
\end{picture}
\caption{\small Planar Legendrian realization of the lollipop graph in (a) standard form $A$ , (b)  standard form $B$.}
\label{standard_vertex_lollipop}
\end{center}
\end{figure}
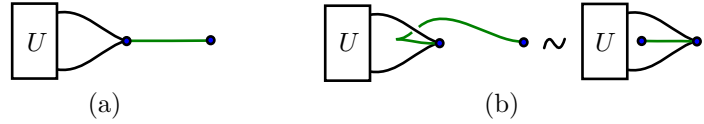

 %----- Figure -----------
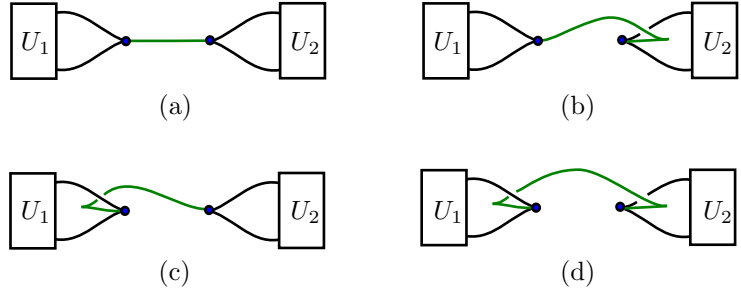
\begin{figure}[htpb!]
\begin{center}
\begin{picture}(320, 90)
\put(0,0){\includegraphics{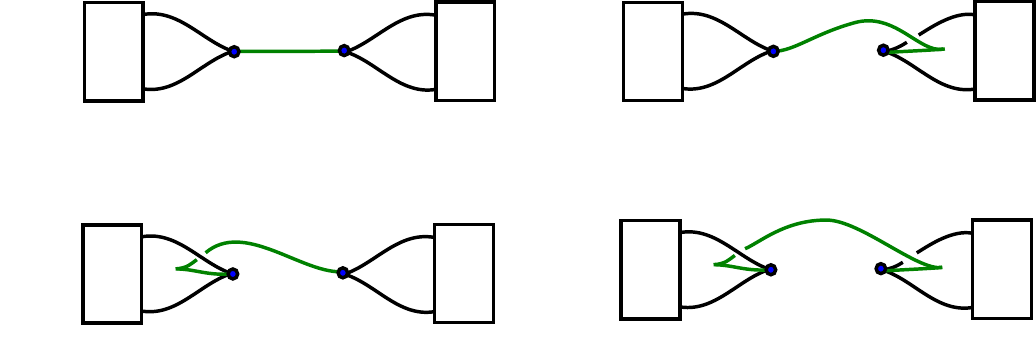}}
\put(80,58){\small (a)}
\put(232,58){\small (b)}
\put(80,-5){\small (c)}
\put(232, -5){\small (d)}
\put(28, 18){\small $U_1$ }
\put(130, 18){\small $U_2$ }
\put(184,18){\small $U_1$}
\put(286, 18){\small $U_2$}
\put(28, 82){\small $U_1$ }
\put(130, 82){\small $U_2$ }
\put(184,82){\small $U_1$}
\put(286, 82){\small $U_2$}
\end{picture}
\caption{\small Planar Legendrian realization of the handcuff graph in (a) standard form $AA$, (b)  standard form $AB$, (c) standard form $BA$, (d)  standard form $BB$.}
\label{standard_vertex}
\end{center}
\end{figure}

In \cite{EF}, Eliashberg and Fraser showed that a Legendrian unknot $K$ is Legendrian isotopic to a unique unknot in standard form. 
The number of cusps and crossings of the unknot in standard form (see Figure \ref{standard_unknotEF}) are uniquely determined by $tb(K)$ and $rot(K)$ as follows:
\begin{enumerate}
\item If $rot(K)\ne 0$ (Figure \ref{standard_unknotEF}(a)), then 
$$ tb(K) = -(2t+1+s)  $$
$$ rot(K) =  \left\{
\begin{array}{rl}
s, & \mbox{ if  the leftmost cusp is a down cusp} \\
-s, & \mbox{ if the leftmost cusp is an up cusp}  \end{array}
\right.  $$
\item If $rot(K) =0$  (Figure \ref{standard_unknotEF}(b)), then 
$$ tb(K) = -(2t+1)  $$
\end{enumerate}

%----- Figure -----------
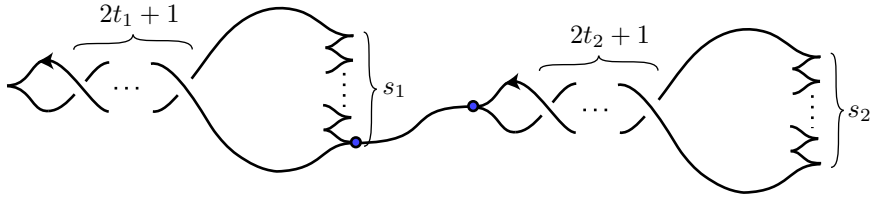
\begin{figure}[htpb!]
\begin{center}
\begin{picture}(320, 72)
\put(0,0){\includegraphics{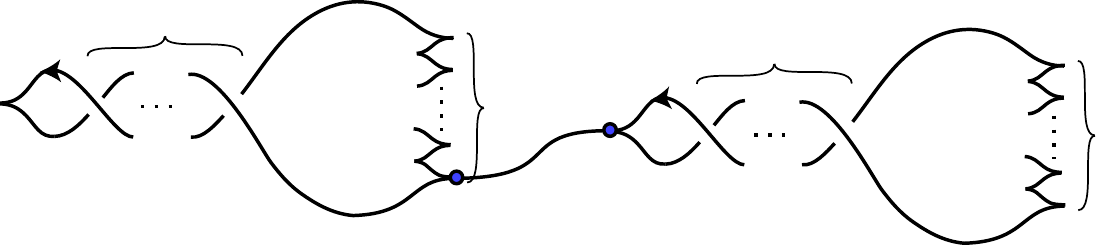}}
\put(35, 65){\small $2t_1+1$}
\put(213, 57){\small $2t_2+1$}
\put(142, 38){\small $s_1$}
\put(318, 29){\small $s_2$}
\end{picture}
\caption{\small Legendrian handcuff graph in standard form $AA$ depicted with both unknotted components with $rot > 0$.}
\label{standard_handcuff}
\end{center}
\end{figure}

\begin{lemma}
Let $G$ be a Legendrian graph consisting of a Legendrian knot and a cut edge connected to it. Through Legendrian isotopy, the cut edge can be moved to be connected at any point of the knot.
\label{slide0}
\end{lemma}
\begin{proof}
We work with a front projection of the graph $G$. Away from the cusps the cut edge can be moved by planar isotopy. 
A cut edge can be passed through a right cusp as in Figure \ref{slide} (below or above, depending on how it sits with respect to the cusp). 
Passing from the lower strand to the upper strand  of a right cusp can be obtained by vertical reflection of the two illustrated cases. 
Diagrams for passing through a left cusp can obtained by horizontal reflection of the diagrams for the right cusp.
\end{proof}

%----- Figure ----------- Slide---

\begin{figure}[htpb!]
\begin{center}
\begin{picture}(430, 130)
\put(0,0){\includegraphics{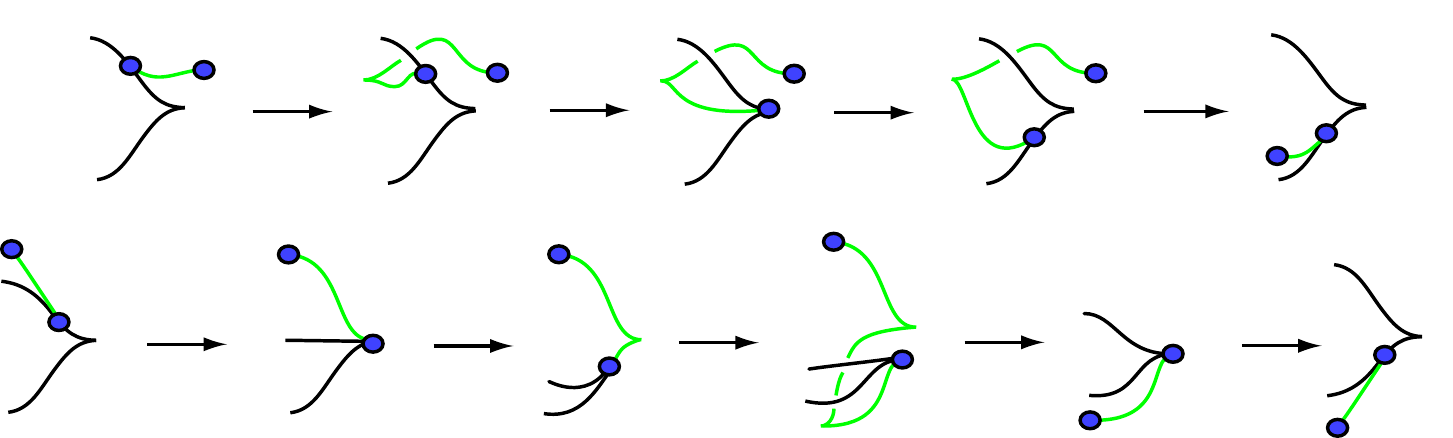}}
\put(75, 98){\small VI}
\put(163, 98){\small IV}
\put(246, 98){\small IV}
\put(330, 98){\tiny retract}
\put(48, 31){\small IV}
\put(129, 31){\small IV}
\put(198, 31){\small VI}
\put(278, 31){\tiny retract}
\put(363, 31){\small IV}
\end{picture}
\caption{\small Sliding the cut edge past a right cusp.}
\label{slide}
\end{center}
\end{figure}
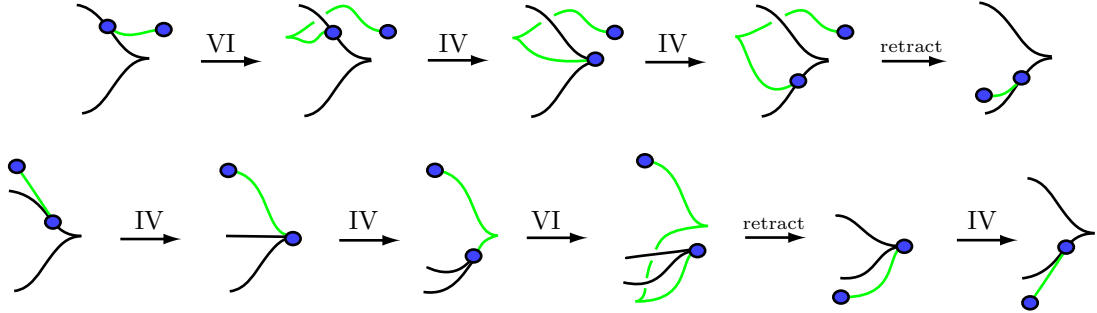

%--------- Theorem -------------- Lollipop

\begin{theorem} 

A pair $(tb, rot)$ determines exactly two Legendrian isotopy classes for a planar Legendrian realization of the lollipop graph.

\end{theorem}

\begin{proof}

We construct a Legendrian isotopy between $L$ and one of the two standard forms. 
Denote by $v_1$ the valence three vertex of $L$, by $v_2$ the valence one vertex of $L$, denote by $U$ the loop edge of $L$, and denote by $e$ the cut edge of $L$.  We work with a front projection of $L$.

\textit{Step 1} (remove the crossings of the cut edge with itself and with $U$).
Starting from $v_2$ towards $v_1$, retract the edge $e$ in a $\de-$neighborhood of $v_1$,  and remove all its self crossings in the front projection, as well as the crossings between $e$ and $U$. 

\textit{Step 2} (put $U$ in standard form). 
Change $L$ by Legendrian isotopy in a neighborhood of $v_1$ such that the unknot $U$ is everywhere smooth. 
By \cite{EF}, there exists a unique unknot in standard form which is Legendrian isotopic to $U$.  
Take $U$ to standard form through Legendrian isotopy, while keeping $v_2$ and its neighborhood containing $e$ away from the isotopy.
We can do this by sliding the cut edge when necessary, as in Lemma \ref{slide0}.

\textit{Step 3} (slide the cut edge to the lower right cusp of $U$).  Using Lemma \ref{slide0}, slide the cut edge so that it connects to the rest of the graph at the lower right cusp of $U$. 
Starting from $v_2$ towards $v_1$, retract the edge $e$ in a $\de-$neighborhood of $v_1$,  and remove all self crossings in the front projection, as well as the all crossings between $e$ and $U$.

Now the graph is in one of the two standard forms. 
Since a standard form of the unknot is uniquely determined by $tb$ and $rot$, each of the standard forms of the lollipop graph are also determined by $tb$ and $rot$.  
Thus we have two Legendrian isotopy classes for the lollipop graph for each pair $(tb, rot)$. 

\end{proof}

%--------- Theorem -------------- handcuff----

\begin{theorem} 
A pair $(tb, rot)$ determines exactly four Legendrain isotopy classes for a planar Legendrian realization of the handcuff graph.
 \end{theorem}

\begin{proof}

We construct a Legendrian isotopy between $L$ and one of the four standard forms.  
Denote by $v_1$ and $v_2$ the two vertices of $L$, denote by $U_1$ and $U_2$ the two loop edges of $L$, and denote by $e$ the cut edge of $L$. 
We work with a front projection of $L$. 

\textit{Step 1} (make $U_1$ and $U_2$ disjoint in the front projection). 
Since the graph $L$ is topologically equivalent to the embedding in Figure \ref{lollipop}(b), the two unknots $U_1$ and $U_2$ bound disks $D_1$ and $D_2$ which are disjoint from each other and disjoint from the rest of the graph. 
Shrink the disks $D_1$ and $D_2$ in $\de-$neighborhoods of $v_1$ and $v_2$, with $\de$ small enough for there to exist no crossings between $U_1$ and $U_2$ in the front projection. 

%----- Figure ----------- flip 1--- outside of cusp
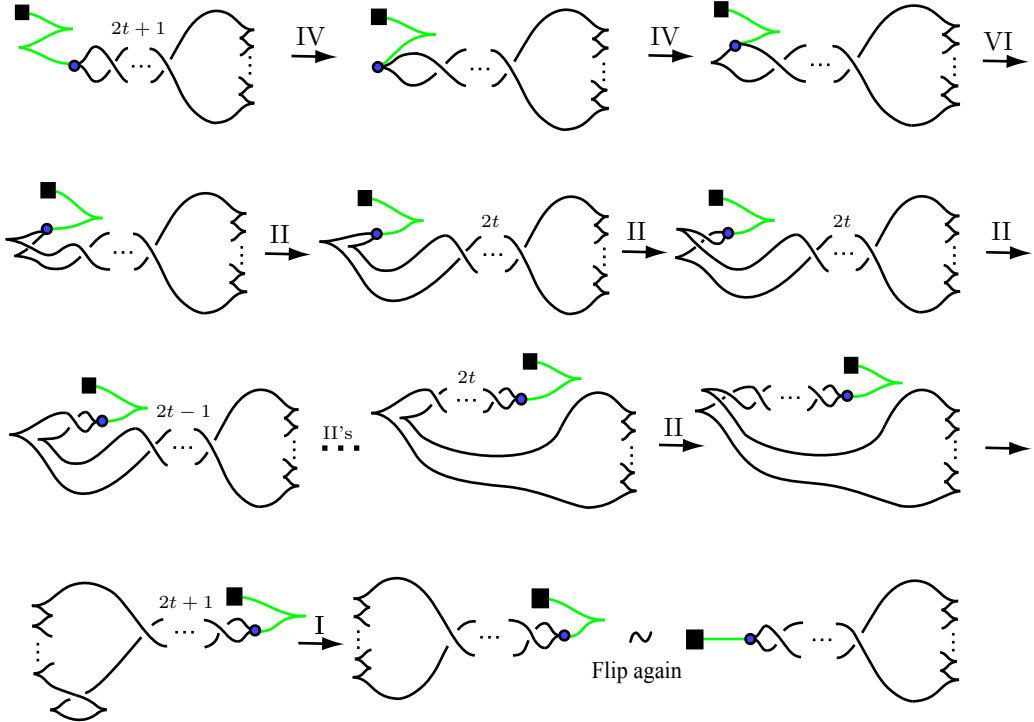
\begin{figure}[htpb!]
\begin{center}
\begin{picture}(390, 280)
\put(0,0){\includegraphics{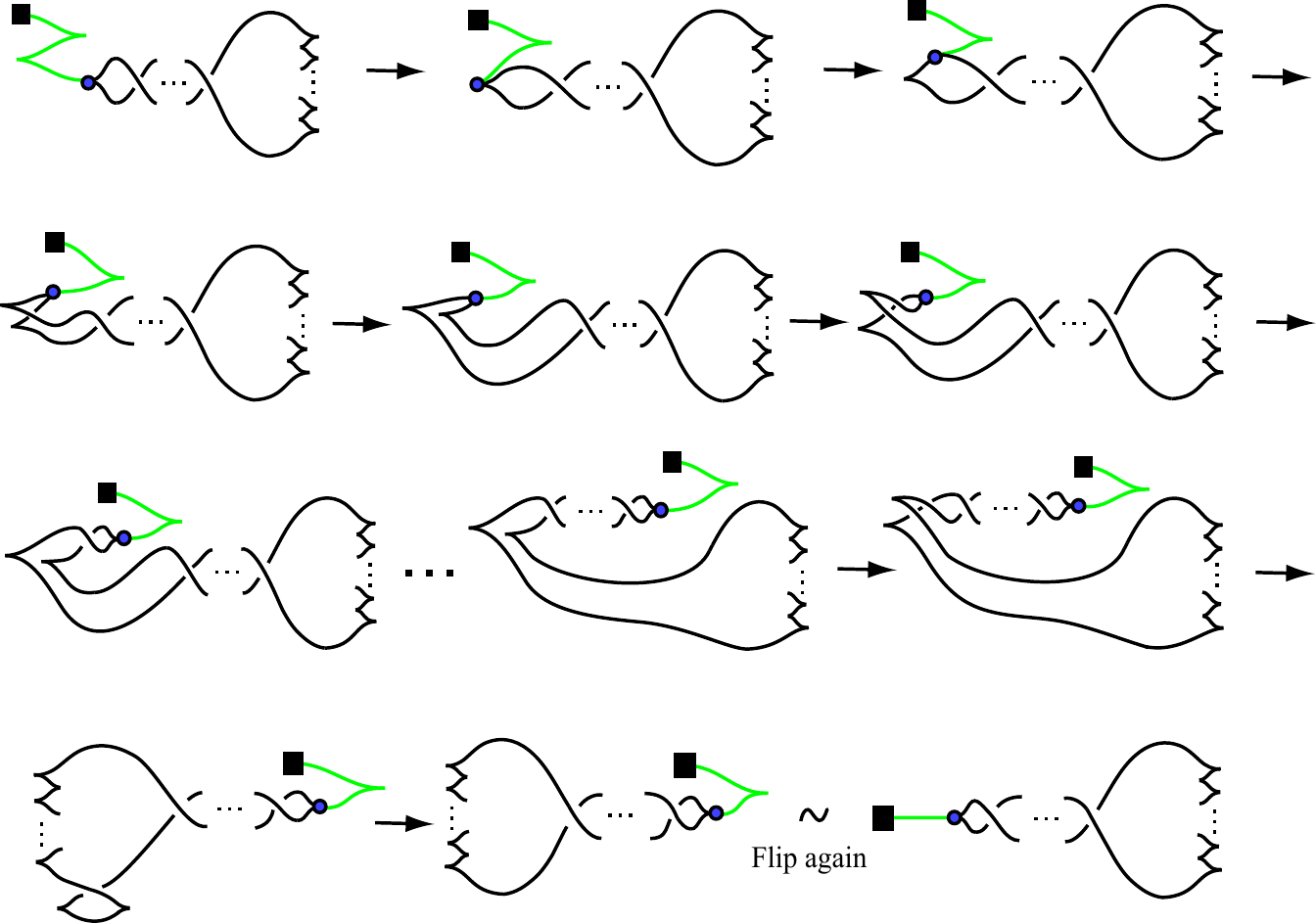}}
\put(110, 256){\small IV}
\put(244, 256){\small IV}
\put(370, 254){\small VI}
\put(100, 181){\small II}
\put(235, 182){\small II}
\put(120, 107){\tiny II's}
\put(373, 182){\small II}
\put(249, 108){\small II}
\put(117, 33){\small I}
\put(40, 260){\tiny $2t+1$}
\put(180, 187){\tiny $2t$}
\put(313, 187){\tiny $2t$}
\put(57, 115){\tiny $2t-1$}
\put(171, 128){\tiny $2t$}
\put(58, 43){\tiny $2t+1$}
\put(117, 33){\small I}

\end{picture}
\caption{\small Flip of $U_2$ undoes the stabilization in the cut edge for the case when the cut edge connects \textit{outside} the cusp at $v_2$. }
\label{flip1}
\end{center}
\end{figure}

\textit{Step 2} (remove crossings of the cut edge with itself and with $U_1$ and $U_2$). 
There exists an embedded $2-$sphere $S_1$ such that $U_1$ is contained in the $3-$ball  $B_1$ bounded by $S_1$ and $S_1$ intersects the cut edge at one point, $w_1$. 
Shrink $B_1$ in a small neighborhood of $w_1$.
Starting from $w_1$ retract the cut edge while carrying along the neighborhood of $w_1$ and undo its knotting outside of $B_1$ as well as all crossings with $U_2$. 

Next, there exists an embedded $2-$sphere $S_2$ disjoint from the sphere $S_1$ such that $U_2$ lies in the $3-$ball $B_2$ bounded by $S_2$ and $S_2$ intersects the cut edge at one point, $w_2$. 
Shrink $B_2$ in a small neighborhood of $w_2$. 
This move may introduce a crossing between $e$ and $U_2$ in the front projection.
Starting from $w_2$ retract the cut edge while carrying along the neighborhood of $w_2$ and undo its knotting outside of $B_2$ as well as all crossings between the cut edge and $U_1$ in the front projection.

\textit{Step 3} (put $U_1$ and $U_2$ in standard form, slide one end of the cut edge to the lower right cusp of $U_1$, and slide the other end to the left cusp of $U_2$).
Take $U_1$ into a small neighborhood of $v_1$.
Modify $U_2$ through a Legendrian isotopy which takes it to the unknot in standard form having the assigned $tb$ and $rot$. 
By sliding the cut edge repeatedly (as in Lemma \ref{slide0}), we can keep $v_1$ and $U_1$ away for this isotopy. 
Once $U_2$ is in standard form, using Lemma \ref{slide0} slide the cut edge so that it connects to $U_2$ at the left cusp of $U_2$. 
The cut edge can sit in two ways with respect to the other two edge segments at this cusp.

Leaving $U_2$ in standard form and leaving $e$ connected to $U_2$ at the left cusp, move $U_1$ and $e$ through Legendrian isotopy so that the front projection of $U_1$ lies outside and to the right of the bounded region in the plane determined by the front projection of $U_2$. 

Modify $U_1$ through a Legendrian isotopy which takes it to the unknot in standard form having the assigned $tb$ and $rot$. 
By sliding the cut edge repeatedly (as in Lemma \ref{slide0}) without contracting it, we leave $U_2$ unchanged.  
Once $U_1$ is in standard form, using Lemma \ref{slide0} slide the cut edge without contracting it so that it connects to $U_1$ at the lower right cusp of $U_1$. 

%----- Figure -----------flip 2---- inside of cusp
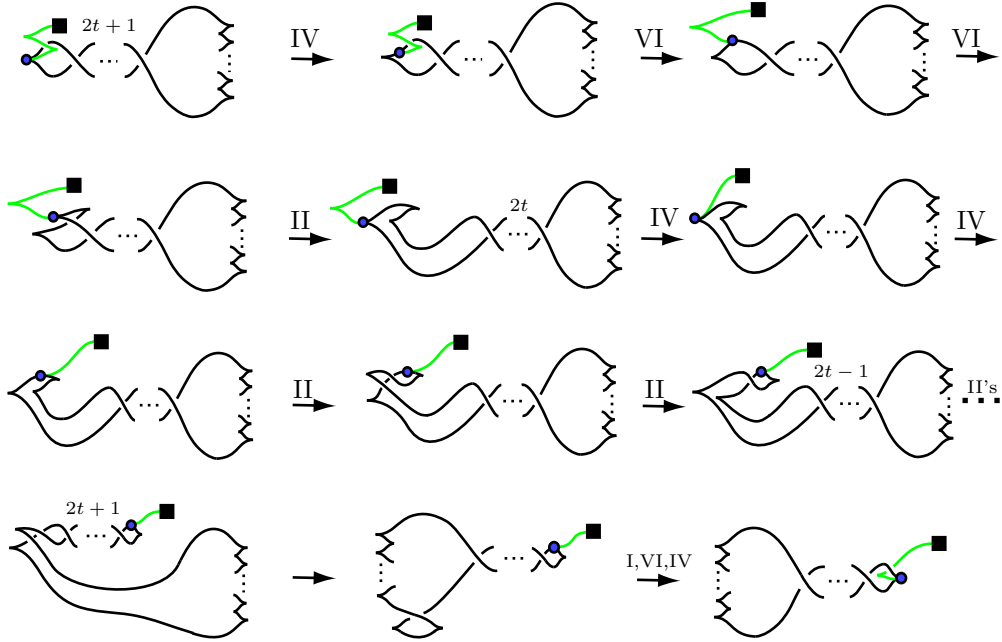
\begin{figure}[htpb!]
\begin{center}
\begin{picture}(380, 240)
\put(0,0){\includegraphics{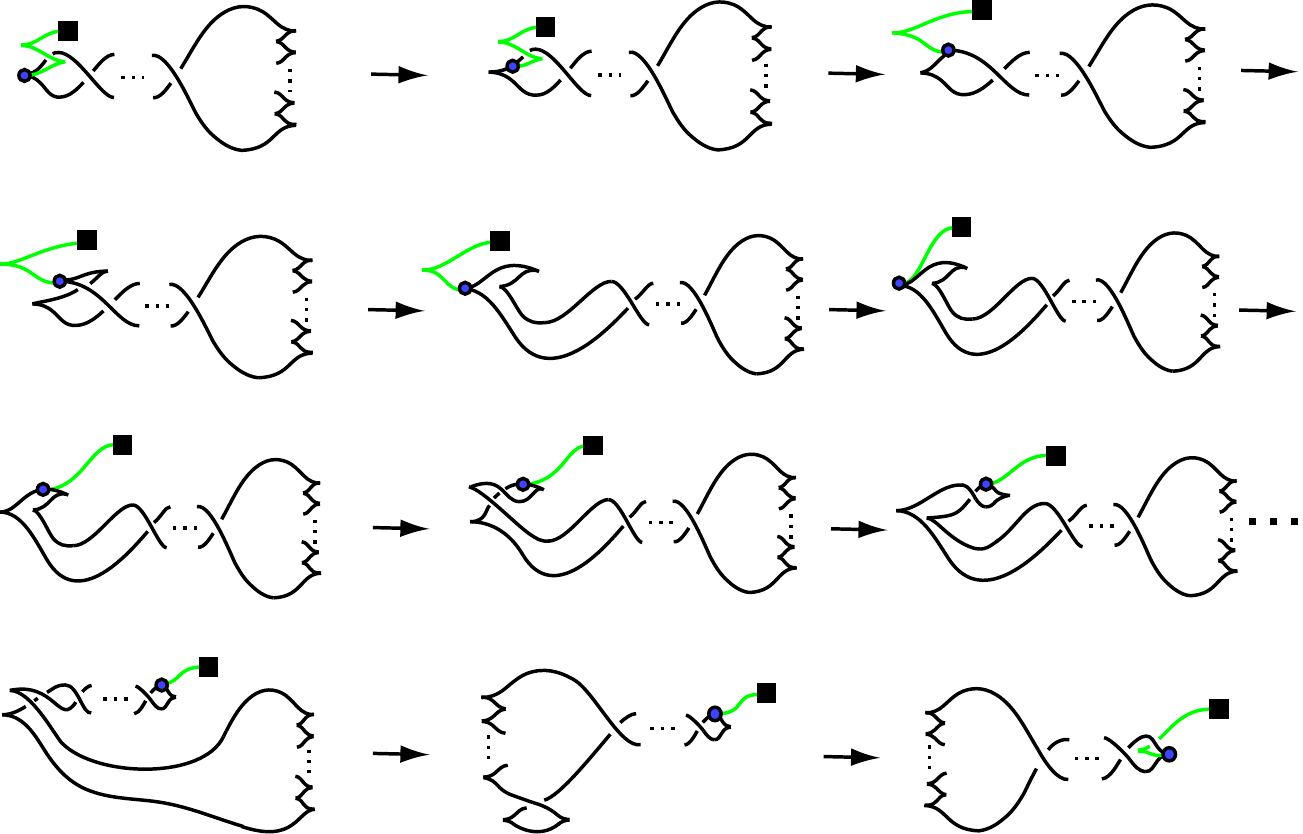}}
\put(107, 224){\small IV}
\put(237, 224){\small VI}
\put(357, 224){\small VI}
\put(107, 155){\small II}
\put(243, 156){\small IV}
\put(359, 155){\small IV}
\put(108, 91){\small II}
\put(241, 91){\small II}
\put(363, 93){\tiny II's}
\put(234, 27){\tiny I,VI,IV}
\put(28, 230){\tiny $2t+1$}
\put(190,162 ){\tiny $2t$}
\put(305, 99){\tiny $2t-1$}
%\put(171, 128){\tiny $2t$}
\put(22, 47){\tiny $2t+1$}
%\put(117, 33){\small I}

\end{picture}
\caption{\small Flip of $U_2$ undoes the stabilization in the cut edge for the case when the cut edge connects \textit{inside} the cusp at $v_2$.}
\label{flip2}
\end{center}
\end{figure}

\textit{Step 4} (undo stabilizations of the cut edge and reach one of the standard forms). 
The cut edge can connect in two ways at  $v_2$ relative to the other two edge segments, \textit{outside} the cusp, or \textit{inside} the cusp.

\begin{enumerate}
\item If the cut edge connects outside the cusp, then the stabilizations of the cut edge can be removed by flipping $U_2$ horizontally, as in Figure \ref{flip1}.
The other type of stabilization is solved by reflecting the diagrams.
After undoing the additional stabilizations, the graph is in one of the standard forms $AA$ or $BA$, depending on how the cut edge sits at $v_1$ relative to the other two edge segments.

\item If the cut edge connects inside the cusp, then the stabilizations of the cut edge can be removed by flipping $U_2$ horizontally, as in Figure \ref{flip2}.
The other type of stabilization is solved by reflecting the diagrams.
After undoing the additional stabilizations, the graph is in one of the standard forms $AB$ or $BB$, depending on how the cut edge sits at $v_1$ relative to the other two edge segments.

\end{enumerate}

\end{proof}

\begin{remark}
If we replace the two unknotted cycles by cycles which are knots whose Legendrian type is determined by $tb$ and $rot$ the theorem still holds.
\end{remark}
%Further, we would like to answer the analog question for the $\theta$-graph and $K_4$. 

%---Bibliography---------------
\bibliographystyle{amsplain}

\end{document}